\documentclass[12pt]{article}

\usepackage{amsmath}
\usepackage{amssymb}
\usepackage{amsfonts}
\usepackage{latexsym}
\usepackage{color}

\catcode `\@=11 \@addtoreset{equation}{section}

\catcode `\@=12

\newcommand{\be}{\begin{equation}}
	\newcommand{\en}{\end{equation}}
\newcommand{\bea}{\begin{eqnarray}}
	\newcommand{\ena}{\end{eqnarray}}
\newcommand{\beano}{\begin{eqnarray*}}
	\newcommand{\enano}{\end{eqnarray*}}
\newcommand{\bee}{\begin{enumerate}}
	\newcommand{\ene}{\end{enumerate}}

\newtheorem{thm}{Theorem}
\newtheorem{cor}[thm]{Corollary}
\newtheorem{lemma}[thm]{Lemma}
\newtheorem{prop}[thm]{Proposition}

\newtheorem{example}[thm]{Example}

\newenvironment{proof}{\noindent {\bf Proof --}}{\hfill$\square$ \vspace{3mm}\endtrivlist}

\catcode `\@=11 \@addtoreset{equation}{section}
\catcode `\@=12

\begin{document}
	
	\thispagestyle{empty}
	
	\begin{center}
		{\Large \bf On Positive Vectors in Indefinite Inner Product Spaces}   \vspace{1cm}\\
		
		{\large Fabio Bagarello}\\
		Dipartimento di Ingegneria,
		Universit\`a di Palermo,\\ I-90128  Palermo, Italy\\
		and I.N.F.N., Sezione di Catania\\
		e-mail: fabio.bagarello@unipa.it\\

		\vspace{2mm}
		
		{\large Sergiusz Ku\.{z}el}\\
		Faculty of Applied Mathematics, AGH University, \\ 30-059 Krak\'ow, Poland \\
		e-mail: kuzhel@agh.edu.pl
		\end{center}
\begin{abstract}
\noindent Let $\mathcal{H}$ be a linear space equipped with an indefinite inner product 
$[\cdot, \cdot]$. Denote by $\mathcal{F}_{++}=\{f\in\mathcal{H} \ : \ [f,f]>0\}$
the nonlinear set of positive vectors in $\mathcal{H}$.
We demonstrate that the properties of a linear operator 
$W$ in $\mathcal{H}$ can be uniquely determined by its restriction to 
$\mathcal{F}_{++}$. In particular, we prove that the bijectivity of 
$W$ on $\mathcal{F}_{++}$  is equivalent to $W$ being {\em close} to a unitary operator with respect to $[\cdot, \cdot]$. Furthermore, we consider a one-parameter semi-group of operators $W_+ = \{W(t) : t \geq 0\}$, where each 
$W(t)$ maps $\mathcal{F}_{++}$   onto itself in a one-to-one manner. We show that, under this natural restriction, the semi-group $W_+$  can be transformed into a one-parameter group 
$U = \{U(t) : t\in\mathbb{R}\}$ of operators that are unitary with respect to $[\cdot, \cdot]$.  By imposing additional conditions, we show how to construct a suitable definite inner product $\langle\cdot, \cdot\rangle$, based on $[\cdot, \cdot]$, which guarantees the unitarity of the operators $U(t)$ in the Hilbert space obtained by completing $\mathcal{H}$
 with respect to $\langle\cdot, \cdot\rangle$. 
\end{abstract}
	
\vspace{1cm}

{\bf Mathematics Subject Classification 2020:}  46C20,47B50, 47D03 

\vspace{1cm}

{\bf Keywords:}	Indefinite inner product space; positive vectors; unitary operators; semi-group of operators.

	\newpage
	\section{Introduction}
	\label{s1}
	Interest in linear spaces with indefinite inner product (indefinite inner product spaces) emerged in the early forties, drawing attention from various directions simultaneously, including both theoretical physicists and mathematicians. Spaces with indefinite inner product in quantum mechanics were first encountered by P. Dirac and W. Pauli \cite{Dir, Pauli}.
	Following these pioneering works, significant interest periodically arose in theoretical physics regarding indefinite inner product spaces   (indefinite metric spaces)  \cite{AM, Gup, Nagy, 123}.
	
	Recently, a renewed source of interest in this field emerged when
	it was proposed in \cite{R2} that the reality of the spectrum of the non-self-adjoint operator $H=-\frac{d^2}{dx^2} + x^2 + ix^3$
	is closely linked to its previously unnoticed
	$\mathcal{PT}$-symmetry. This symmetry allows $H$  to be considered a Hermitian operator with respect to the indefinite inner product ($\mathcal{PT}$ inner product)
	\begin{equation}\label{uman4}
		[f, g]=\int_{-\infty}^{\infty}[\mathcal{PT}f(x)]g(x)dx=\int_{-\infty}^{\infty}\overline{f(-x)}g(x)dx,
	\end{equation}
	where $(\mathcal{P}f)(x)=f(-x)$ is the space reflection (parity) operator and $(\mathcal{T}f)(x)=\overline{f(x)}$ is the complex conjugation operator.
	
	$\mathcal{PT}$-Symmetric approach has attracted continuous interest from physicists towards developing a variant of quantum mechanics based on Hamiltonians that are not Hermitian but satisfy certain less restrictive and more physically motivated symmetry conditions (for the $\mathcal{PT}$-symmetric approach in quantum mechanics, see  \cite{book} and references therein). In most cases, such kind of  non-self-adjoint Hamiltonians can be interpreted as Hermitian ones with respect to an appropriately chosen indefinite inner product \cite{AK_AK, Tanaka}.
	
	The common feature among the aforementioned activities is the utilization of a non-Hermitian Hamiltonian $H$ that is symmetric with respect to an indefinite inner product $[\cdot, \cdot]$ within a linear space $\mathcal{H}$, which is endowed with \emph{a given topology}, typically the Hilbert space topology. At times, establishing a specific topology at the outset can yield surprising outcomes \cite{Kuz, AK_Siegl}. This is because a newly constructed positive-definite inner product defined by $[\cdot, \cdot]$ may not align well with the original topology.

	In this paper, we aim to develop an alternative approach to overcome such kind of  difficulties.
	We start by considering a linear space $\mathcal{H}$
	equipped with an indefinite inner product $[\cdot,\cdot]$, \emph{without making any assumptions about its Hilbert space topology}.
	However, when attempting to develop a consistent quantum theory for a Hamiltonian
	$H$ that is Hermitian with respect to
	$[\cdot, \cdot]$, we encounter the fundamental difficulty of ensuring the unitarity of the dynamics generated by $H$. The problem  deals with the potential existence of eigenstates
	$f$  such that~
	$[f, f]<0$.
	Due to the probabilistic interpretation of the norm of states in standard quantum theory, the presence of eigenstates with negative indefinite inner product norm raises interpretational issues. A common approach in the literature (as reviewed in \cite{book}) to address the issue of interpretation involves constructing a special operator
	$\mathcal{C}$ such that the new $\mathcal{CPT}$ inner product
	\begin{equation}\label{dur}
		\langle f, g\rangle_{\mathcal{CPT}}=\int_{-\infty}^{\infty}[\mathcal{CPT}f(x)]g(x)dx=[\mathcal{C}f, g]
	\end{equation}
	turns out to be positive definite, and  $H$ becomes Hermitian in a Hilbert space defined by the standard (i.e., positive definite) inner product $\langle \cdot, \cdot\rangle_{\mathcal{CPT}}$.
	It should be noted that the construction of operator $\mathcal{C}$
	in an explicit form is quite difficult because
	$\mathcal{C}$
	depends on the specific Hamiltonian $H$
	being studied. It is not surprising that the majority of the available formulae are approximate, usually expressed as the leading terms of some perturbative series \cite{book}.
	
	In this paper, we present an alternative approach to establishing the unitarity of the dynamics generated by a time-independent non-self-adjoint operator
	$H$.   Our method involves constructing a suitable positive definite inner product, denoted by
	$\langle\cdot,\cdot\rangle$, from an initially indefinite inner product
	$[\cdot, \cdot]$. The goal is to ensure that
	$H$
	becomes Hermitian with respect to $\langle\cdot,\cdot\rangle$. In this case, the dynamics governed by the time-dependent Schrödinger equation can be characterized by a one-parameter group of  operators
	$\{ U(t)=e^{iHt} \ :
	\ t\in\mathbb{R}\}$
	that are unitary with respect to $\langle\cdot,\cdot\rangle$. According to Stone's theorem, the Hermiticity of
	$H$ and the unitarity of $\{U(t)\}$ are equivalent concepts.
	For this reason, due to certain technical aspects, we focus on analyzing a certain one-parameter semi-group
	$\{W(t) : t\geq{0}\}$. Our goal is to demonstrate that, under natural conditions, the semi-group
	$\{W(t)\}$
	can be extended into a one-parameter group of operators $\{U(t)\}$, and a suitable inner product
	$\langle\cdot,\cdot\rangle$  can be constructed, ensuring the unitarity of the operators $\{U(t)\}$.

	We start with the natural assumption that physical states
	$f$ must have a positive norm,
	$[f,f]>0$. Consequently, in the initial stage, we restrict our consideration to the action of linear operators on the \emph{nonlinear set} of positive  vectors\footnote{We adopt the conventional notation
		$\mathcal{F}_{++}$
		from the theory of indefinite inner product spaces \cite{AK_Azizov, AK_Bognar} to represent the set of positive vectors.}
	$\mathcal{F}_{++}=\{f\in\mathcal{H} \ : \ [f,f]>0\}.$
	Our work in this area was inspired by {a brief remark in \cite{AK_Bender5} during the study of manifestly non-self-adjoint $\mathcal{PT}$-symmetric Hamiltonians: `One can try to formulate a quantum theory associated with $\mathcal{PT}$-symmetric Hamiltonians by insisting that physical states must have positive norm. This leads to a quantum mechanics defined on a nonlinear state space. Such investigations are of interest, but the existence of negative-norm eigenstates still leaves open serious interpretational issues.'}
	
	The nonlinear set $\mathcal{F}_{++}$ defines the entire linear space $\mathcal{H}$
	in the sense that every
	$f\in\mathcal{H}$
	can be decomposed as
	$f=f_1+f_2$,
	where
	$f_i\in\mathcal{F}_{++}$. It should be noted that this decomposition can not be unique. However,
	a linear operator $W$ defined on $\mathcal{F}_{++}$ can be uniquely extended to a linear operator $W$ acting on  $\mathcal{H}$
	(Lemma \ref{uf5b}).
	
	In Section \ref{s2}, we consider a one-parameter semi-group of operators  $W_+=\{W(t) : t\geq{0}\}$ acting on a linear space $\mathcal{H}$ with indefinite inner product $[\cdot, \cdot]$. Our primary assumption regarding the behavior of $W(t)$
	is as follows: since physical states must have positive norms, we will assume that 
 \begin{center}
	\emph{the operators $W(t)$
	map the set of positive vectors $\mathcal{F}_{++}$
	onto itself in a one-to-one manner.}
	\end{center}
	This somehow natural assumption allows the semi-group of operators
		$W_+$
		to be transformed into a one-parameter group of operators $U=\{U(t) : t\in\mathbb{R}\}$ (see subsection \ref{2.3}).
	The operators $U(t)$ are unitary with respect to $[\cdot, \cdot]$, meaning they are bijective on $\mathcal{H}$ and preserve the indefinite inner product, i.e., $[U(t)f, U(t)f]=[f,f]$ for all $f\in\mathcal{H}$.
	However, such a group of unitary operators with respect to $[\cdot, \cdot]$
	cannot always be realized as  unitary operators {with respect to a positive definite inner product $\langle\cdot,\cdot\rangle$}.  We investigate this problem in Section \ref{3}.
	The first natural question is:
	\emph{how  can we determine a `relevant' positive definite inner product $\langle\cdot,\cdot\rangle$  on
		$\mathcal{H}$ using the indefinite inner product
		$[\cdot,\cdot]$?}
	This problem is addressed in subsection \ref{s3.1}. The solution depends on whether
	${\mathcal{H}}$ is decomposable, in the sense given above. In the decomposable case, a definite inner product can be defined in a natural way, see (\ref{uf2}) and (\ref{uf3}).
	
	In the non-decomposable case, the problem becomes significantly more complex due to the large number of neutral vectors. For the remainder of the paper, we will focus exclusively on decomposable spaces, except for few comments in subsection \ref{s3.1}
	
	Each decomposable space  $\mathcal{H}$ admits infinitely many fundamental decompositions (see (\ref{uf2}))
	$$
	\mathcal{H}=\mathcal{L}_+[\dot{+}]\mathcal{L}_-,
	$$
	into mutually orthogonal positive and negative linear subspaces $\mathcal{L}_\pm$.
	Every fundamental decomposition is  determined by an operator of fundamental symmetry  $J_{\mathcal{L}}$  that acts on elements
	$f\in\mathcal{H}$ as follows:
	$$
	J_{\mathcal{L}}f=J_{\mathcal{L}}(f_++f_-)=f_+-f_-, \qquad f_\pm\in\mathcal{L}_\pm.
	$$
	The concept of fundamental symmetry operators is closely related to the notion of operators $\mathcal{C}$  in the context of $\mathcal{PT}$-symmetric approach in quantum mechanics \cite{book}.
	
	The fundamental decomposition $\mathcal{H}=\mathcal{L}_+[\dot{+}]\mathcal{L}_-$,
	enables the definition of a \emph{definite} inner product $\langle{\cdot}, \cdot\rangle_{\mathcal{L}}=[J_{\mathcal{L}}\cdot, \cdot]$ on
	$\mathcal{H}$ that is a generalization of $\mathcal{CPT}$ inner product, (\ref{dur}). The completion of $\mathcal{H}$ with respect to the norm $\|\cdot\|_{\mathcal{L}}=\sqrt{\langle{\cdot}, \cdot\rangle_{\mathcal{L}}}$ results in both a Hilbert space $(\mathcal{H}_{ext},\langle{\cdot}, \cdot\rangle_{\mathcal{L}})$ and a Krein space $\mathcal{H}_{ext}$ with the indefinite inner product $[\cdot,\cdot]$. At first glance, the space $\mathcal{H}_{ext}$ seems like a suitable choice for our group of operators $U$, initially defined on the dense subset $\mathcal{H}$. However, we cannot be certain that these operators can be continuously extended to the entire space $\mathcal{H}_{ext}$. This problem is completely solved in Corollary \ref{uf64b} using operators of fundamental symmetry $J_t$, which correspond to the fundamental decompositions
	$$
	\mathcal{H}=\mathcal{L}_+^t[\dot{+}]\mathcal{L}_-^t, \qquad \mathcal{L}_+^t=U(t)\mathcal{L}_+, \quad  \mathcal{L}_-^t=U(t)\mathcal{L}_-.
	$$
	These decompositions show how the original decomposition (\ref{uf2}) changes under the action of the operators from the group $U$.
	
	If the conditions of Corollary \ref{uf64b} hold, meaning the operators $J_t$ are bounded for all
	$t>0$, then the group $U$ can be continuously extended to a group of unitary operators in the Krein space $\mathcal{H}_{ext}$, i.e., they  are unitary with respect to the indefinite inner product $[\cdot,\cdot]$. However, this result is not entirely satisfactory, as our goal is to obtain a group of operators that are unitary with respect to a definite inner product, ensuring unitarity in a Hilbert space. According to Theorem \ref{uf53}, the Hilbert spaces case is characterized by an additional condition of uniform boundedness imposed on the set of fundamental symmetries $\{J_t\}_{t>0}$. In this case, the set of equivalent definite inner products generated by various fundamental decompositions of the Krein space $\mathcal{H}_{ext}$ includes a `right' inner product $\langle \cdot, \cdot \rangle_{\mathcal{M}}$, making the group $U_{ext}$ a group of unitary operators in the Hilbert space $(\mathcal{H}_{ext}, \langle \cdot, \cdot \rangle_{\mathcal{M}})$.

	Throughout the paper, the symbol $\mathcal{H}$ denotes a linear space. A linear bijective  operator $H$ on $\mathcal{H}$ is referred as \emph{fully invertible}. A fully invertible operator $H$ acting on $\mathcal{H}$ is called \emph{unitary} with respect to an indefinite inner product $[\cdot, \cdot]$ if $[Hf,Hg]=[f,g]$ for all $f, g\in\mathcal{H}$.
	The symbol $\dot{+}$ denotes the direct sum of sets, whereas $[\dot{+}]$ is employed when the sets comprising the direct sum are mutually orthogonal under the indefinite inner product $[\cdot, \cdot]$.

	\section{Indefinite inner product spaces without topology}\label{s2}
	\subsection{Indefinite inner product space.}\label{s2.1}
	Let $\mathcal{H}$ be a complex linear space with a
	Hermitian sesquilinear form $[\cdot,\cdot]$  i.e., a mapping $[\cdot,\cdot]:{\mathcal H}\times{\mathcal H}\to\mathbb{C}$ such that  $[f,\alpha_1g_1+\alpha_2g_2]=\alpha_1[f,g_1]+\alpha_2[f,g_2]$ and $[f,g]=\overline{[g,f]}$ for all $f, g, g_1, g_2\in{\mathcal H}$, $\alpha_1, \alpha_2\in\mathbb{C}$.
	A Hermitian sesquilinear form $[\cdot,\cdot]$ is referred to as \emph{an indefinite inner product} if there exist vectors 
$f\in\mathcal{H}$ such that 
$[f,f]>0$, and simultaneously there exist vectors 
$g\in\mathcal{H}$ for which 
$[g, g]<0$.

	A  space ${\mathcal H}$  equipped with an indefinite inner product $[\cdot, \cdot]$ is called an \emph{indefinite inner product space}.
	An indefinite inner product space $\mathcal{H}$ is called
	\emph{non-degenerate} if the equality
	$[f, g]=0$ for all $g\in\mathcal{H}$ implies that $f=0$.
	
	Each degenerate indefinite inner product space $\mathcal{H}$
		contains a non-zero linear subspace of isotropic vectors
		$\mathcal{X}=\{f\in\mathcal{H} : [f, g]=0, \ g\in\mathcal{H}\}$. The space $\mathcal{H}$ can be transformed into a non-degenerate space by considering the quotient space
		$\mathcal{H}/\mathcal{X}$. Taking this into account, from now on \emph{we will only consider non-degenerate indefinite inner product spaces}. 
  
	A  vector $f\in\mathcal{H}$ is called {\it positive, negative}, or {\it neutral}
	if
	$$
	[f,f]>0,\qquad [f,f]<0,  \quad \mbox{or} \quad [f,f]=0.
	$$
	
	Denote by $\mathcal{F}_{++}$ and  $\mathcal{F}_{--}$ the sets of \emph{positive} and \emph{negative} vectors of $\mathcal{H}$, respectively.
	It is evident that if $f\in\mathcal{F}_{++}$ (or $\mathcal{F}_{--}$), then $c{f}\in\mathcal{F}_{++}$  (or $\mathcal{F}_{--}$) for any $c\in\mathbb{C}$.
	However, the sum of two positive (or negative) elements, $f+g$, does not necessarily belong to
	$\mathcal{F}_{++}$ (or $\mathcal{F}_{--}$).
	In this context, the sets $\mathcal{F}_{++}$ and $\mathcal{F}_{--}$
	are inherently \emph{nonlinear}, and therefore cannot form subspaces within an indefinite inner product space $\mathcal{H}$ \cite[Corollary 2.7]{AK_Bognar}.
	
	\begin{example} Consider a linear space $\mathcal{H}=\mathbb{R}^4$ with indefinite inner product (Minkowski inner product)
		$$
		[f_1, f_2]=t_1t_2-x_1x_2-y_1y_2-z_1z_2, \qquad f_j=(t_j, x_j,  y_j, z_j).
		$$
		The set of positive vectors $\mathcal{F}_{++}=\{f=(t, x, y, z) : x^2+y^2+z^2<t^2\}$ in $\mathcal{H}$  often appears when dealing with the Minkowski spacetime.
	\end{example}

	An operator
	$W$ defined on
	$\mathcal{F}_{++}$
	is said to be \emph{linear} if it satisfies the following conditions:
	$W(c{f})=c{Wf}$ for $f\in\mathcal{F}_{++}$ and $c\in\mathbb{C}$, and $W(f+g)=Wf+Wg$ for any $f,g\in\mathcal{F}_{++}$ such that $f+g\in\mathcal{F}_{++}$.
	
	\begin{lemma}\label{uf5b}
		A space $\mathcal{H}$ allows for the following decomposition:
		$
		\mathcal{H}=\mathcal{F}_{++}+\mathcal{F}_{++}$.  A linear operator $W$ defined on $\mathcal{F}_{++}$ can be uniquely extended to a linear operator acting on the entire space $\mathcal{H}$.
	\end{lemma}
	\begin{proof}
		Let $f\in\mathcal{H}$ and $f_+\in\mathcal{F}_{++}$. Then
		\begin{equation}\label{uman1}
			[f+c{f_+}, f+c{f_+}]=[f, f]+2 Re[f,cf_+]+|c|^2[f_+, f_+]>0
		\end{equation}
		for sufficiently large $|c|$.
		{Let
			${C}_{f_+}$
			denote the set of all $c\in\mathbb{C}$
			for which inequality (\ref{uman1}) holds. By construction, ${C}_{f_+}$
			contains a neighborhood of infinity.}
		
		If $c\in{C}_{f_+}$, then the  vector $g_+(c)=f+c{f_+}$ belongs to $\mathcal{F}_{++}$ and
		\begin{equation}\label{uman45}
			f=g_+(c)+(-c{f_+}).
		\end{equation}
	This formula justifies the presentation of $\mathcal{H}$ as the algebraic sum $\mathcal{H}=\mathcal{F}_{++}+\mathcal{F}_{++}$.
		
		{Now, if $W$ is defined on $\mathcal{F}_{++}$, we can naturally put with the use of (\ref{uman45}) that}    \begin{equation}\label{uman2}
			Wf:=Wg_+(c)-cWf_+, \qquad c\in{C}_{f_+}
		\end{equation}
		
		{Formula (\ref{uman2}) provides a unique linear extension of $W$ onto $\mathcal{H}$ under the condition that
			(\ref{uman2}) is well-posed}, i.e., the vector $Wf$ does not depend on the choice of $c$ and $f_+$, which is what we will prove next.
		
		{Let
			$c'\in{C}_{f_+}$ and $c'\not=c$. Then, (\ref{uman2}) determines a vector $(Wf)':=Wg_+(c')-c'Wf_+$}, where $g_+(c')=f+c'{f_+}$.
		{Taking into account that $g_+(c)-g_+(c')=(c-c')f_+\in\mathcal{F}_{++}$ we obtain
			$$
			Wg_+(c)-Wg_+(c')=W(g_+(c)-g_+(c'))=(c-c')Wf_+.
			$$
			Hence,
			$Wf-(Wf)'=Wg_+(c)-Wg_+(c')+(c'-c)Wf_+=(c-c')Wf_+-(c-c')Wf_+=0$.
			This means that} the vector $Wf$ in (\ref{uman2}) \emph{does not depend} on the choice $c\in{C}_{f_+}$.
		
		Now assume that a different vector $\tilde{f}_+\in\mathcal{F}_{++}$ is used in (\ref{uman1}) - (\ref{uman2}) instead of $f_+$ and put $\widetilde{(Wf)}:=W\tilde{g}_+(\tilde{c})-\tilde{c}W\tilde{f}_+$, where $\tilde{g}_+(\tilde{c})=f+\tilde{c}{\tilde{f}_+}$ and $\tilde{c}\in{C}_{\tilde{f}_+}$. Then, for all $c,\tilde{c}\in{C}_{f_+}\cap{C}_{\tilde{f}_+}$,
		\begin{equation}\label{uman3}
			Wf-\widetilde{(Wf)}=Wg_+(c)-W\tilde{g}_+(\tilde{c})+W\tilde{c}\tilde{f}_+-Wcf_+.
		\end{equation}
		The set ${C}_{f_+}\cap{C}_{\tilde{f}_+}$ is non empty and it contains a neighborhood of infinity. This means we can select the parameters
			$c$ and $\tilde{c}$ in (\ref{uman3}) such that $
			\tilde{c}\tilde{f}_+-cf_+\in\mathcal{F}_{++}$. {Then $\tilde{g}_+(\tilde{c})-g_+(c)=\tilde{c}\tilde{f}_+-cf_+\in\mathcal{F}_{++}$.} In this case 
		$$
		W\tilde{c}\tilde{f}_+-Wcf_+=W(\tilde{c}\tilde{f}_+-cf_+) \quad \mbox{and} \quad Wg_+(c)-W\tilde{g}_+(\tilde{c})=W(g_+(c)-\tilde{g}_+(\tilde{c}))=-W(\tilde{c}\tilde{f}_+-cf_+).
	$$ 
 Substituting these relations into (\ref{uman3}) yields $Wf-\widetilde{(Wf)}=0$ that completes the proof.
	\end{proof}

	\subsection{Operators preserving the positivity of vectors.}\label{s2.2}

	By Lemma \ref{uf5b}, a linear operator originally defined on the nonlinear set of positive vectors $\mathcal{F}_{++}$ can be uniquely extended to a linear operator $W$ acting on the entire indefinite inner product space $\mathcal{H}$. This implies that the behavior of $W$ on $\mathcal{F}_{++}$ uniquely determines all of its properties on the linear space $\mathcal{H}$.
	
	By imposing additional conditions on the behavior of $W$ on $\mathcal{F}_{++}$, one can further specify the properties of its linear extension to the full space. In particular, the following statement illustrates that the bijectivity of $W$ on $\mathcal{F}_{++}$ is equivalent to $W$ being {\em close} to a unitary operator with respect to the indefinite inner product $[\cdot,\cdot]$.
	\begin{thm}\label{ups30}
		Let $W$ be a linear operator  defined on the set of positive vectors $\mathcal{F}_{++}$  and
		such that $W$ maps $\mathcal{F}_{++}$ onto itself in a one-to-one manner. Then its extension to the linear space $\mathcal{H}$ is fully invertible and
		\begin{equation}\label{e1a}
			[Wf, Wg]=\theta[f, g], \qquad f, g\in\mathcal{H},
		\end{equation}
		where
		%\begin{equation}\label{e2a}
		$$	\theta=\inf_{f\in\mathcal{F}_{++}}\frac{[Wf,Wf]}{[f,f]}>0.
  $$
		%\end{equation}
	\end{thm}
	\begin{proof}
		By virtue of Lemma \ref{uf5b}, there exists a unique linear extension of
		$W$ to the entire linear space $\mathcal{H}$. This extension preserves the property that $W$ maps
		$\mathcal{F}_{++}$ bijectively into itself. Applying \cite[Theorem II.8.9]{AK_Bognar} (or, the original version \cite[Theorem 1.5]{Krein}), we complete the proof.
	\end{proof}
	
	{\bf Remark:--}
		Verifying the  condition, that
		$W$ maps $\mathcal{F}_{++}\to\mathcal{F}_{++}$ bijectively, can be reduced to an equivalent condition: there exists the inverse operator $W^{-1}$ and the operators
		$W$ and $W^{-1}$ both map $\mathcal{F}_{++}$
		into $\mathcal{F}_{++}$.

	\begin{example}\label{ex4} 
		Let $\mathcal{H}=\mathcal{S}(\mathbb{R})$ be the set of smooth functions with rapid decay
		equipped with $\mathcal{PT}$-inner product (\ref{uman4})
		The operator $W$ defined by the formula
		$$
		(Wf)(x)=f(rx), \qquad r\neq0, \quad f\in\mathcal{S}(\mathbb{R})
		$$
		maps  $\mathcal{F}_{++}$   into $\mathcal{F}_{++}$ since
		$$
		[Wf, Wf]=\int_{-\infty}^{\infty}\overline{f(-rx)}\,f(rx)dx=\frac{1}{|r|}[f,f], \qquad f\in\mathcal{S}(\mathbb{R}).
		$$
		The similar property holds for its inverse $(W^{-1}f)(x)=f(r^{-1}x)$, i.e., $W^{-1} : \mathcal{F}_{++} \to \mathcal{F}_{++}$. This means that the mapping $W$ determines bijection on $\mathcal{F}_{++}$ and Theorem \ref{ups30} can be applied with $\theta=\frac{1}{|r|}$.
	\end{example}

	The condition that $W$
	acts as a bijection on
	$\mathcal{F}_{++}$
	is quite strong, and this is also reflected in the spectral properties of
	$W$. To illustrate this phenomenon we consider a finite linear span  $\mathcal{H}={\sf \ span}\{f_n\}_{n=1}^\infty$, where $\{f_n\}_{n=1}^\infty$ is a sequence of vectors orthogonal with respect to $[\cdot, \cdot]$, i.e., $[f_n, f_m]=0$, $n\not={m}$. The assumption that $\mathcal{H}$ is a non-degenerate space leads to the requirement that $[f_n, f_n]\not=0$.
	Without loss of generality we can assume that $|[f_n, f_n]|=1$ for all $n\in\mathbb{N}$.
	Separating the sequence  $\{f_n\}$ by the signs of $[f_n,f_n]$:
	$$
	f_{n}=\left\{\begin{array}{l}
		f_{n}^+ \quad \mbox{if} \quad [f_{n},f_{n}]=1, \\
		f_{n}^- \quad \mbox{if} \quad [f_{n},f_{n}]=-1,
	\end{array}\right.
	$$
	we obtain two sequences of positive $\{f_n^+\}$ and negative $\{f_m^-\}$  elements, which also satisfy $[f_{n}^+,f_{m}^-]=0$.

	\begin{prop}\label{uf46}
		Let $W$ be a linear operator acting in an indefinite inner product space $\mathcal{H}={\sf \ span}\{f_n\}_{n=1}^\infty$ and defined on vectors $f_n$ as follows:
		$$
		Wf_n=\mu_nf_n, \qquad \mu_n\in\mathbb{C}, \quad \mu_n\neq0.
		$$
		The following are equivalent:
		\begin{itemize}
			\item[(i)] the operator $W$ maps $\mathcal{F}_{++}$ onto itself in a one-to-one manner;
			\item[(ii)] there exists $c>0$ such that $
			c=|\mu_1|=|\mu_2|=\ldots = |\mu_n|=\ldots
			$ for all  $\mu_n$.
		\end{itemize}
	\end{prop}
	\begin{proof} $(i)\to(ii)$. By virtue of  Theorem \ref{ups30}, the  relation (\ref{e1a})
		holds. Setting here $f_n=f=g$  we obtain $(ii)$ with $c=\sqrt{\theta}>0$.
		
		$(ii)\to(i)$. An arbitrary $f\in\mathcal{H}\cap\mathcal{F}_{++}$
		has the form $f=\sum_{i=1}^pc_{n_i}f_{n_i}^++\sum_{i=1}^kc_{m_i}f_{m_i}^-$, where $\{f_{n_i}^+\}_{i=1}^p$ and $\{f_{m_i}^-\}_{i=1}^k$ are arbitrary chosen vectors from the sets $\{f_n^+\}$ and  $\{f_m^-\}$, respectively and the coefficients $\{c_{n_i}, c_{m_i}\}$
		satisfy the relation
		$$
		[f, f]=\sum_{i=1}^p|c_{n_i}|^2-\sum_{i=1}^k|c_{m_i}|^2>0.
		$$
		If condition $(ii)$ holds, then
		$$
		[Wf, Wf]=\left[\sum_{i=1}^p\mu_{n_i}c_{n_i}f_{n_i}^++\sum_{i=1}^k\mu_{m_i}c_{m_i}f_{m_i}^-, \ \sum_{i=1}^p\mu_{n_i}c_{n_i}f_{n_i}^++\sum_{i=1}^k\mu_{m_i}c_{m_i}f_{m_i}^-\right]=
		$$
		$$
		=|c|^2\left(\sum_{i=1}^p|c_{n_i}|^2-\sum_{i=1}^k|c_{m_i}|^2\right)=|c|^2[f,f].
		$$
		This means that
		$W$ maps $\mathcal{F}_{++}$ into itself.
		The similar result holds true for the inverse operator $W^{-1}$ defined by the formula  $W^{-1}f_n=\frac{1}{\mu_n}f_n$.
		Therefore, $W$ maps $\mathcal{F}_{++}$ onto itself in a one-to-one manner.
\end{proof}
{\bf Remark:--} The indefiniteness of the inner product $[\cdot, \cdot]$ plays a crucial role in our analysis. If $[\cdot, \cdot]$ were a definite inner product, then $\mathcal{F}_{++}=\mathcal{H}\setminus\{0\}$, and, obviously, Proposition \ref{uf46} would not hold. For instance, let us consider a strictly increasing, positive and bounded sequence $\{\mu_n\}_{n=1}^\infty$.  Let us further consider a vector space 
	$\mathcal{H}={\sf span}\{g_n\}_{n=1}^\infty$, in which $\langle g_n,g_m\rangle=\delta_{nm}$, $\langle\cdot,\cdot\rangle$ being a definite scalar produce, and a linear operator $V$ defined on $\mathcal{H}$ as $Vg_n=\gamma_ng_n$. It is clear that $V$ is invertible, but the equality in $(ii)$ of Proposition \ref{uf46} does not hold. Of course, for the pre-Hilbert space $\mathcal{H}={\sf span}\{g_n\}_{n=1}^\infty$ the notion of $\mathcal{F}_{++}$ is not so relevant.

 \subsection{Construction of the group of operators.}\label{2.3}
	We begin with a one-parameter semi-group of linear operators $W_+=\{W(t) : t\geq{0}\}$, i.e.,
	\begin{equation}\label{add1}
	W(0)=I, \quad W(t_1+t_2)=W(t_1)W(t_2)=W(t_2)W(t_1), \quad t_1, t_2\in[0, \infty)
	\end{equation}
	acting in an indefinite inner product space $\mathcal{H}$
	and such that
	{\sf all operators of $W_+$  map the set of positive vectors $\mathcal{F}_{++}$ onto itself in a one-to-one manner}. Then, as follows from Theorem \ref{ups30}, the operators $W(t)$ are fully invertible in $\mathcal{H}$ and the relation
	\begin{equation}\label{e1}
		[W(t)f, W(t)g]=\theta(t)[f, g], \qquad f, g\in\mathcal{H}, \quad t\geq{0}
	\end{equation}
	holds, where
	\begin{equation}\label{e2}
		\theta(t)=\inf_{f\in\mathcal{F}_{++}}\frac{[W(t)f,W(t)f]}{[f,f]}>0.
	\end{equation}
	\begin{thm}\label{thm1}
		If the function $r(t)=[W(t)f, W(t)f]$ is continuous on $[0, \infty)$ for some  $ f\in\mathcal{F}_{++}$,
		then the function $\theta(\cdot)$ in (\ref{e2}) has the form $\theta(t)=e^{\alpha{t}}$, where $\alpha\in{\mathbb{R}}$.
	\end{thm}
	\begin{proof} For each $t_1, t_2>0$, equations  (\ref{add1}) and (\ref{e1}) imply that, $\forall f\in\mathcal{H}$,
		$$
		\theta(t_1+t_2)[f,f]=[W(t_1+t_2)f,W(t_1+t_2)f]=[W(t_2)W(t_1)f,W(t_2)W(t_1)f]=$$
		$$=\theta(t_2)[W(t_1)f,W(t_1)f]=\theta(t_2)\theta(t_1)[f,f].
		$$
		In particular, if we take $f\in\mathcal{F}_{++}$, $[f,f]>0$, so that  $\theta(t_1+t_2)=\theta(t_1)\theta(t_2)$. Moreover, it follows from (\ref{add1}) - (\ref{e2}) that $\theta(0)=1$, is continuous on $[0, \infty)$, and real. Hence, by virtue of \cite[Theorem 4.17.2]{Hille}, there  exists $\alpha\in\mathbb{R}$ such that  $\theta(t)=e^{\alpha{t}}$.

	\end{proof}
	
	Denote
	\begin{equation}\label{e4}
		U(t)=e^{-\frac{\alpha}{2}t}W(t), \qquad t\geq{0}.
	\end{equation}
	Operators $U(t)$ are fully invertible on $\mathcal{H}$ and  $U_+=\{U(t) : t\geq{0}\}$   is a one-parameter semi-group of linear operators acting in $\mathcal{H}$. Moreover, in view of (\ref{e1}) 
\begin{equation}\label{uf12}
		[U(t)f, U(t)g]=[f, g], \qquad f, g\in\mathcal{H}, \quad t\geq{0}.
	\end{equation}
	{This means that operators $U(t)$ are unitary with respect to $[\cdot, \cdot]$.}
	
	By defining $U(-t)$ as the inverse of $U(t)$ for $t>0$, we can extend the semi-group $U_+$ to  a one-parameter group of  fully invertible linear operators $U=\{U(t) : t\in\mathbb{R}\}$ {that are unitary with respect to $[\cdot,\cdot]$.}
	\begin{example}\label{uman10}
		Consider the semi-group of operators
		$$
		(W(t)f)(x)=f(e^tx), \qquad f\in\mathcal{S}(\mathbb{R}), \quad t\geq0
		$$
		acting in the space $\mathcal{H}=\mathcal{S}(\mathbb{R})$ with
		$\mathcal{PT}$ inner product
		(\ref{uman4}).
		By virtue of Example
		\ref{ex4}, the operators
		$W(t)$ are bijections on
		$\mathcal{F}_{++}$ and the function $\theta(t)$ in (\ref{e2}) coincides with
		$e^{-t}$. According to (\ref{e4}) and (\ref{uf12}), the operators
		$U=\{U(t) : t\in\mathbb{R}\}$ defined as
		$$
		(U(t)f)(x)=e^{t/2}W(t)f=e^{t/2}f(e^tx), \quad f\in\mathcal{H}=\mathcal{S}(\mathbb{R}), \quad t\in\mathbb{R}
		$$
		constitute a one-parameter  dilation group on $\mathcal{H}$ that preserves the indefinite inner product $[\cdot,\cdot]$.    \end{example}

	\begin{example}\label{uman63}
		Consider an indefinite inner product space $\mathcal{H}={\sf \ span}\{f_n  :  n\in\mathbb{N}\}$, where $\{f_n\}$ is a sequence of vectors orthonormal with respect to $[\cdot, \cdot]$, i.e., $|[f_n, f_m]|=\delta_{nm}$.
		The formula
		$$
		W(t)f_n=e^{\lambda_nt}f_n, \qquad \lambda_n\in\mathbb{C}, \quad t\geq{0},
		\quad n\in\mathbb{N}.
		$$
		determines a one-parameter semi-group of linear operators on $\mathcal{H}$.
		
		In view of Proposition \ref{uf46} {with $\mu_n=e^{\lambda_nt}$}, the operator $W(t)$ maps $\mathcal{F}_{++}$ onto itself in a one-to-one manner if and only if
		\begin{equation}\label{uman1963}
			0<c(t)=e^{{\sf Re}\lambda_1t}=e^{{\sf Re}\lambda_2t}=\ldots=e^{{\sf Re}\lambda_nt}=\ldots
		\end{equation}
		In this case, the function $\theta(\cdot)$ in (\ref{e2}) has the form $\theta(t)=c^2(t)=e^{\alpha{t}}$, where
		$$
		\alpha={2{\sf Re}\lambda_1}={2{\sf Re}\lambda_2}=\ldots={2{\sf Re}\lambda_n}=\ldots
		$$
		By virtue of  Theorem \ref{thm1}  and the relation (\ref{e4}),
		the semi-group $U_+$ of operators $U(t)$ on $\mathcal{H}$
		is defined by the formula
		\begin{equation}\label{ff1}
			U(t)f_n=e^{-\frac{\alpha}{2}}W(t)f_n=e^{-{\sf Re} \lambda_n}W(t)f_n=e^{i{\sf Im} \lambda_nt}f_n, \quad t\geq{0}.
		\end{equation}
		The semi-group $U_+$ can be extended to a group {$U$ of unitary operators with respect to $[\cdot,\cdot]$} by considering (\ref{ff1}) for all $t\in\mathbb{R}$.
	\end{example}
{\bf Remark:--} {All the results in this section can be reformulated for the case of negative vectors, $\mathcal{F}_{--}$.}

	\section{One-parameter group of unitary operators in a Hilbert space} \label{3}
	The group
	$U$ of unitary operators with respect to $[\cdot, \cdot]$ considered in subsection \ref{2.3}
	cannot always be guaranteed to be realized as unitary operators in a Hilbert space. Below, we investigate this problem in detail.
	To begin, we consider in subsection \ref{s3.1} a more general problem: \emph{how can a definite inner product $\langle\cdot,\cdot\rangle$  be defined on an indefinite inner product space
		$\mathcal{H}$ using the indefinite inner product
		$[\cdot, \cdot]$?}
	
	\subsection{Decomposable and non-decomposable spaces.}\label{s3.1}
	A linear subspace $\mathcal{L}$ of an indefinite inner product space $\mathcal{H}$ is called \emph{positive} or \emph{negative} if $\mathcal{L}\subset\mathcal{F}_{++}\cup\{0\}$ or $\mathcal{L}\subset\mathcal{F}_{--}\cup\{0\}$, respectively. A linear subspace $\mathcal{L}$ is called  \emph{neutral} if all vectors of $\mathcal{L}$ are neutral.
	
	A space $\mathcal{H}$  is said to be \emph{decomposable}, if it can be represented as the orthogonal (with respect to $[\cdot,\cdot]$)
	direct sum
	\begin{equation}\label{uf2}
		\mathcal{H}=\mathcal{L}_+[\dot{+}]\mathcal{L}_-,
	\end{equation}
	where $\mathcal{L}_+$ is a positive subspace of $\mathcal{H}$ and $\mathcal{L}_-$ is a negative subspace of $\mathcal{H}$.
	
	The decomposition (\ref{uf2}) of
	$\mathcal{H}$ into mutually orthogonal positive and negative linear subspaces is not unique, and each such decomposition is referred to as \emph{fundamental}.
	
	A decomposable space $\mathcal{H}$ is called a \emph{Krein space} if
	the positive subspace $\mathcal{L}_+$ is complete with respect to the inner product $[\cdot, \cdot]$, and likewise, the negative subspace $\mathcal{L}_-$ possesses the same property with respect to the inner product $-[\cdot, \cdot]$. In other words, both $\mathcal{L}_+$ and $\mathcal{L}_-$, endowed with the inner products $[\cdot, \cdot]$ and $-[\cdot, \cdot]$, respectively, must be Hilbert spaces. In general, \emph{we do not assume that the original space $\mathcal{H}$ is a Krein space}.
	
	Decomposable spaces are frequently employed in mathematical physics. Their primary advantage lies in the fact that the decomposition (\ref{uf2}) allows us to determine a natural, \emph{definite} inner product\footnote{The subscript $\mathcal{L}$ is added in the notation of the inner product to emphasize its connection with the decomposition (\ref{uf2}).} on $\mathcal{H}$:
	\begin{equation}\label{uf3}
		\langle f, g\rangle_{\mathcal{L}}=[f_+, g_+]-[f_-, g_-], \quad f=f_++f_-, \quad g=g_++g_-, \quad f_\pm, g_\pm\in\mathcal{L}_\pm
	\end{equation}
	that is closely related to $[\cdot, \cdot]$.
	The space
	$\mathcal{H}$ endowed with the definite inner product (\ref{uf3}) is a pre-Hilbert space.
	
	The decomposition (\ref{uf2}) is uniquely determined by an operator $J_{\mathcal{L}}$ in $\mathcal{H}$ that acts on elements
	$f\in\mathcal{H}$ as follows:
	$$
	J_{\mathcal{L}}f=J_{\mathcal{L}}(f_++f_-)=f_+-f_-, \qquad f_\pm\in\mathcal{L}_\pm.
	$$
	The operator $J_{\mathcal{L}}$ is called the \emph{operator of fundamental symmetry} associated with the fundamental decomposition (\ref{uf2}).
	This operator is fully invertible in $\mathcal{H}$, its spectrum  consists of two eigenvalues $-1$ and $1$, and the corresponding kernel subspaces coincide with $\mathcal{L_\pm}$, namely
	$$
	\mathcal{L}_+=\ker(I-J_{\mathcal{L}})=(I+J_{\mathcal{L}})\mathcal{H} \quad \mbox{and} \quad \mathcal{L}_-=\ker(I+J_{\mathcal{L}})=(I-J_{\mathcal{L}})\mathcal{H}.
	$$
	
	It is important that, the definite inner product $\langle\cdot, \cdot\rangle_{\mathcal{L}}$ on $\mathcal{H}$ is  defined by $J_{\mathcal{L}}$ and the indefinite inner product $[\cdot,\cdot]$, specifically
	\begin{equation}\label{bu1}
		\langle{f}, g\rangle_{\mathcal{L}}=[J_{\mathcal{L}}f, g] \quad \mbox{and} \quad
		[f, g]=\langle{J}_{\mathcal{L}}{f}, g\rangle_{\mathcal{L}}, \qquad f, g\in\mathcal{H}.
	\end{equation}
	
	The inner product $\langle\cdot, \cdot\rangle_{\mathcal{L}}$  depends on the choice of the subspaces $\mathcal{L}_\pm$ in (\ref{uf2}), which are not determined uniquely.
	Changing $\mathcal{L}_\pm$ results in the construction of various definite inner products and norms, which \emph{are  equivalent
		in the case of a Krein space} $\mathcal{H}$. However, if $\mathcal{H}$ is not a Krein space, distinct decompositions (\ref{uf2}) can result in \emph{non-equivalent norms} on $\mathcal{H}$. Let's take a closer look at this significant point and consider an another decomposition of
	$\mathcal{H}$
	\begin{equation}\label{uf2b}
		\mathcal{H}=\mathcal{M}_+[\dot{+}]\mathcal{M}_-,
	\end{equation}
	where $\mathcal{M}_+$ is a positive subspace of $\mathcal{H}$ and $\mathcal{M}_-$ is a negative subspace of $\mathcal{H}$.

	Denote by $\|\cdot\|_{\mathcal{X}}=\sqrt{\langle\cdot, \cdot\rangle_{\mathcal{X}}}$
	the norm on $\mathcal{H}$ that is generated by the definite inner product
	$\langle\cdot, \cdot\rangle_{\mathcal{X}}$, \ $\mathcal{X}\in\{\mathcal{L}, \mathcal{M}\}$.
	It is easy to see that the operator of fundamental symmetry $J_{\mathcal{X}}$
	satisfies the relations
	\begin{equation}\label{uf220}
		J_{\mathcal{X}}^2f=f, \quad  \|J_{\mathcal{X}}f\|_\mathcal{X}=\|f\|_\mathcal{X}, \quad  [J_{\mathcal{X}}f, g]=[f,J_{\mathcal{X}}g], \quad f, g\in\mathcal{H}.
	\end{equation}

	In what follows, when it is necessary to specify the inner product being used on linear space
	$\mathcal{H}$
	we will use the notation $(\mathcal{H}, \langle\cdot, \cdot\rangle_{\mathcal{X}})$.
	
	\begin{thm}\label{ups131b}The following are equivalent:
		\begin{itemize}
			\item[(i)] the norms $\|\cdot\|_{\mathcal{L}}$ and $\|\cdot\|_{\mathcal{M}}$ are equivalent on $\mathcal{H}$;
			\item[(ii)] the operator of fundamental symmetry $J_{\mathcal{M}}$ associated with the fundamental decomposition (\ref{uf2b}) is bounded in the pre-Hilbert space
			$(\mathcal{H}, \langle\cdot, \cdot\rangle_{\mathcal{L}})$;
			\item[(iii)] the operator of fundamental symmetry $J_{\mathcal{L}}$ associated with the fundamental decomposition (\ref{uf2}) is bounded in the pre-Hilbert space
			$(\mathcal{H}, \langle\cdot, \cdot\rangle_{\mathcal{M}})$.
		\end{itemize}
	\end{thm}
	\begin{proof}
		$(i) \to (ii)$. If the norms are equivalent, then there exist constants
		$\alpha_i>0$ such that $\|\cdot\|_{\mathcal{L}}\leq\alpha_1\|\cdot\|_{\mathcal{M}}$
		and $\|\cdot\|_{\mathcal{M}}\leq\alpha_2\|\cdot\|_{\mathcal{L}}$.  Using the second relation in (\ref{uf220}), we obtain
		$$
		\|J_{\mathcal{M}}f\|_{\mathcal{L}}\leq\alpha_1\|J_{\mathcal{M}}f\|_{\mathcal{M}}=\alpha_1\|f\|_{\mathcal{M}}\leq\alpha_1\alpha_2\|f\|_{\mathcal{L}}, \qquad f\in\mathcal{H}.
		$$
		Therefore, $J_{\mathcal{M}}$  is bounded in
		$(\mathcal{H}, \langle\cdot, \cdot\rangle_{\mathcal{L}})$.
		The implication $(i) \to (iii)$  is established in a similar manner.
		
		$(ii) \to (i)$. Denote by $\mathcal{H}_{\mathcal{X}}$ the Hilbert space obtained by completing the pre-Hilbert space $(\mathcal{H}, \langle\cdot, \cdot\rangle_{\mathcal{X}})$, \ $\mathcal{X}\in\{\mathcal{L}, \mathcal{M}\}$ .
		Since $|[f, f]|\leq\|f\|_{\mathcal{X}}$ for $f\in\mathcal{H}$, the indefinite inner product $[\cdot,\cdot]$ can also be extended on $\mathcal{H}_{\mathcal{X}}$ by the continuity. The Hilbert space $\mathcal{H}_{\mathcal{X}}$, equipped with the indefinite inner product $[\cdot,\cdot]$ is a Krein space.
		
		By virtue of  $(ii)$, the operator $J_\mathcal{M}$  can be continuously extended to the Hilbert space $\mathcal{H}_{\mathcal{L}}$.
		We will use the same notation $J_\mathcal{M}$
		for this extension. Define the subspaces of $\mathcal{H}_{\mathcal{L}}$:
		$$
		\mathcal{M}_+'=(I+J_\mathcal{M})\mathcal{H}_\mathcal{L}, \qquad  \mathcal{M}_-'=(I-J_\mathcal{M})\mathcal{H}_\mathcal{L}
		$$
		By the construction,  $\mathcal{M}_\pm'$ are extensions of the subspaces $\mathcal{M}_\pm$ from the decomposition (\ref{uf2b}) of $\mathcal{H}$. Note that the relations (\ref{uf220})
		admit the extension on $\mathcal{H}_{\mathcal{L}}$, i.e., they hold true for $f\in\mathcal{H}_{\mathcal{L}}$ and for $J_{\mathcal{X}}=J_{\mathcal{M}}$. Using this extended version of (\ref{uf220}) it is easy to see that $\mathcal{M}_+'$ and $\mathcal{M}_-'$ are orthogonal with respect to $[\cdot, \cdot]$ and $\mathcal{M}_+'\cap\mathcal{M}_-'=\{0\}$.
		Moreover,
		\begin{equation}\label{uf63}
			\mathcal{M}_+'[+]\mathcal{M}_-'=\mathcal{H}_{\mathcal{L}}
		\end{equation}
		since an arbitrary $f\in\mathcal{H}_{\mathcal{L}}$ admits the decomposition
		$f=\frac{1}{2}(I+J_\mathcal{M})f+\frac{1}{2}(I-J_\mathcal{M})f=f_++f_-$, where
		the vectors $f_\pm=\frac{1}{2}(I\pm{J}_{\mathcal{M}})f$ belong to $\mathcal{M}_\pm'$.

		By comparing formulas
		$\mathcal{M}_+'=(I+J_\mathcal{M})\mathcal{H}_\mathcal{L}$ and $\mathcal{M}_+=(I+J_\mathcal{M})\mathcal{H}$, we conclude that
		$\mathcal{M}_+'$
		is a completion of $\mathcal{M}_+$
		in the Hilbert space
		$\mathcal{H}_{\mathcal{L}}$.
		Therefore, $\mathcal{M}_+'$ should be a nonnegative subspace with respect to the indefinite inner product $[\cdot, \cdot]$ (since $\mathcal{M}_+$ is positive). Let $g\in\mathcal{M}_+'$ be a neutral vector. Then $[g, f]=0$ for all $f\in\mathcal{M}_+'$.
		In view of (\ref{uf63}), the last equality holds for every $f\in\mathcal{H}_{\mathcal{L}}$. Hence $g=0$ and the subspace $\mathcal{M}_+'$ is positive with respect to $[\cdot,\cdot]$. Similar reasoning indicates that $\mathcal{M}_-'$
		is a negative subspace.
		
		The considerations above demonstrate that the decomposition (\ref{uf63}) consists of positive and negative subspaces.  This decomposition, similar to (\ref{uf3}), defines the norm $\|\cdot\|_{\mathcal{M}}$
		on $\mathcal{H}_{\mathcal{L}}$. As was shown in \cite{AK_AK}, the norm defined in such a way is equivalent to the norm $\|\cdot\|_{\mathcal{L}}$.
		The implication $(iii) \to (i)$  is established in a similar manner.
	\end{proof}
	
	Denote by $\mathcal{H}_{\mathcal{L}}$ and $\mathcal{H}_{\mathcal{M}}$ the Hilbert spaces obtained by completing the pre-Hilbert spaces $(\mathcal{H}, \langle\cdot, \cdot\rangle_{\mathcal{L}})$ and
	$(\mathcal{H}, \langle\cdot, \cdot\rangle_{\mathcal{M}})$,  respectively. Generally, $\mathcal{H}_{\mathcal{L}}\not=\mathcal{H}_{\mathcal{M}}$  and these Hilbert spaces  coincide if the norms $\|\cdot\|_{\mathcal{L}}$ and $\|\cdot\|_{\mathcal{M}}$ are equivalent on $\mathcal{H}$. 
	
	Assume that the norms $\|\cdot\|_{\mathcal{L}}$ and $\|\cdot\|_{\mathcal{M}}$ are equivalent on $\mathcal{H}$ and
	denote $\mathcal{H}_{ext}=\mathcal{H}_{\mathcal{L}}=\mathcal{H}_{\mathcal{M}}$.
	By virtue of Theorem \ref{ups131b}, the operators $J_\mathcal{X}$, $\mathcal{X}\in\{\mathcal{L}, \mathcal{M}\}$ can be continuously extended to $\mathcal{H}_{ext}$.
	We will use the same notation $J_\mathcal{X}$
	for this extension.
	\begin{prop}\label{ups1}
		If the norms $\|\cdot\|_{\mathcal{L}}$ and $\|\cdot\|_{\mathcal{M}}$ are equivalent on $\mathcal{H}$, then
		there exists a bounded self-adjoint operator $Q$ in $(\mathcal{H}_{ext}, \langle\cdot, \cdot\rangle_{\mathcal{X}})$ such that  $J_{\mathcal{X}}Q=-QJ_{\mathcal{X}}$ and
		$$
		J_{\mathcal{L}}J_{\mathcal{M}}=e^Q, \qquad J_{\mathcal{M}}J_{\mathcal{L}}=e^{-Q}
		$$
	\end{prop}
	\begin{proof}
		By applying (\ref{bu1})  twice (once for
		$J_\mathcal{L}$ and once for $J_{\mathcal{M}}$), we obtain, for $f,g\in\mathcal{H}$,
		\begin{equation}\label{bu2}
			\langle{f}, g\rangle_{\mathcal{L}}=[J_\mathcal{L}f, g]=\langle{J_{\mathcal{M}}}J_{\mathcal{L}}{f}, g\rangle_{{\mathcal{M}}}, \quad \langle{f}, g\rangle_{\mathcal{M}}=[J_{\mathcal{M}}f, g]=\langle{J_{\mathcal{L}}}J_{\mathcal{M}}{f}, g\rangle_{{\mathcal{L}}}
		\end{equation}
		
		In view of Theorem \ref{ups131b},  the relations (\ref{bu2}) can be extended for all $f,g\in\mathcal{H}_{ext}$.  This means that
		$J_{\mathcal{L}}J_{\mathcal{M}}$ is a bounded positive self-adjoint operator in $(\mathcal{H}_{ext}, \langle\cdot, \cdot\rangle_{\mathcal{L}})$ with bounded inverse. Therefore, its spectrum $\sigma(J_{\mathcal{L}}J_{\mathcal{M}})$ is a bounded set  contained within $(0, \infty)$. Using the functional calculus for self-adjoint operators (see, e.g. \cite[Section 5]{Konrad}), one can define a bounded self-adjoint operator $Q=\ln{J_{\mathcal{L}}J_{\mathcal{M}}}$ in $(\mathcal{H}_{ext}, \langle\cdot, \cdot\rangle_{\mathcal{L}})$ so that $J_{\mathcal{L}}J_{\mathcal{M}}=e^Q$.  Then $J_{\mathcal{M}}J_{\mathcal{L}}=e^{-Q}$ since $J_{\mathcal{L}}J_{\mathcal{M}}J_{\mathcal{M}}J_{\mathcal{L}}=J_{\mathcal{M}}J_{\mathcal{L}}J_{\mathcal{L}}J_{\mathcal{M}}=I.$
		
		The space $\mathcal{H}_{ext}$ is a Krein space with the indefinite inner product $[\cdot,\cdot]$ and the completion
		of the subspaces\footnote{It does not matter which norm $\|\cdot\|_{\mathcal{X}}$ we use for making a completion, since they are equivalent in $\mathcal{H}_{ext}$.} from the decompositions
		(\ref{uf2}) and (\ref{uf2b}) give rise to the fundamental decompositions of the Krein space $\mathcal{H}_{ext}$ described by the
		fundamental symmetries $J_{\mathcal{L}}$ and $J_{\mathcal{M}}$. Using now \cite[Theorem 6.2.1]{AK_AK} with $J=J_{\mathcal{L}}$ and $\mathcal{C}=J_{\mathcal{M}}$ we obtain that $Q$ is a bounded self-adjoint operator in the Hilbert space $(\mathcal{H}_{ext}, \langle\cdot, \cdot\rangle_{\mathcal{L}})$ such that $J_{\mathcal{L}}Q=-QJ_{\mathcal{L}}$.
		
		Similarly, by applying \ \cite[Theorem 6.2.1]{AK_AK} with $J=J_{\mathcal{M}}$ and $\mathcal{C}=J_{\mathcal{L}}$, and recalling that $J_{\mathcal{M}}J_{\mathcal{L}}=e^{-Q}$  we arrive at the conclusion that $Q$ is a bounded self-adjoint operator in the Hilbert space $(\mathcal{H}_{ext}, \langle\cdot, \cdot\rangle_{\mathcal{M}})$ such that $J_{\mathcal{M}}Q=-QJ_{\mathcal{M}}$.
	\end{proof}
	
	The following example illustrates a variety of possible decompositions  (\ref{uf2b}) for the simplest case where $\mathcal{H}=\mathbb{C}^2$.
	
	\begin{example}\label{uuu1} Consider an inner product space $\mathcal{H}=\mathbb{C}^2$ with the indefinite inner product
		$$
		[f, g]=\overline{x}_1{y}_2+\overline{x}_2{y}_1, \qquad f=(x_1, x_2), \ g=(y_1, y_2), \quad x_i, y_i\in\mathbb{C}.
		$$
		The space $\mathcal{H}$, endowed with $[\cdot, \cdot]$, is a Krein space. The fundamental decomposition (\ref{uf2}) can be chosen as
		$$
		\mathcal{L}_+=\{f_+=(x, x) : x\in\mathbb{C}\}, \quad \mathcal{L}_-=\{f_-=(x, -x) : x\in\mathbb{C}\}.
		$$
		It is  important to stress that our choice of the initial decomposition (\ref{uf2}) is entirely arbitrary: other choices for $\mathcal{L}_\pm$ can be done, and in fact, as it will appear clearly in what follows, the purpose of this example is to illustrate the vast range of possible decompositions one could adopt given an indefinite inner product space $\mathcal{H}$.

		The corresponding operator of fundamental symmetry ${J}_{\mathcal{L}}$ is defined by
		the Pauli matrix $\sigma_1=\left[\begin{array}{cc}
			0 & 1 \\
			1 & 0 \end{array}\right]$ (we identify operators in $\mathbb{C}^2$ with  $(2\times{2})$-matrices, i.e., $J_\mathcal{L}=\sigma_1$). By virtue of (\ref{bu1}), the inner product $\langle\cdot, \cdot\rangle_{\mathcal{L}}$
		has the form
		$
		\langle{f}, g\rangle_{\mathcal{L}}=[\sigma_1{f},g]=\overline{x}_1{y}_1+\overline{x}_2{y}_2.
		$
		
		Since in finite dimensional Hilbert spaces all norms are equivalent, the assumption of Proposition \ref{ups1} are satisfied, with $J_{\mathcal{L}}=\sigma_1$. Hence we obtain that
		$J_{\mathcal{M}}=\sigma_1e^{Q}$, where $Q$ is a bounded self-adjoint operator in $(\mathcal{H}, \langle\cdot, \cdot\rangle_{\mathcal{L}})$, which anticommutes with $\sigma_1$.
		Let us decompose $Q$ with respect to the basis of Pauli matrices i.e.,
		$$
		Q=q_0\sigma_0+q_1\sigma_{1}+q_2\sigma_{2}+q_3\sigma_3,  \qquad q_j\in\mathbb{C},
		$$
		where $\sigma_0$ is the identity matrix on $\mathbb{C}^2$, $\sigma_2=\left[\begin{array}{cc}
			0 & -i \\
			i & 0 \end{array}\right]$ and $\sigma_3=\left[\begin{array}{cc}
			1 & 0 \\
			0 & -1 \end{array}\right]$.
		The condition of self-adjointness of $Q$ in $(\mathcal{H}, \langle\cdot, \cdot\rangle_{\mathcal{L}})$ implies that $q_j\in\mathbb{R}$.
		Furthermore, the anticommutation with $\sigma_1$ is possible only for $q_0=q_1=0$.
		Hence,
		$$
		Q=q_2\sigma_2+q_3\sigma_3=\rho({\cos\xi}\sigma_{2}+\sin\xi\sigma_3),  \quad \xi\in[0, \pi], \quad \rho\in\mathbb{R},
		$$
		where $\rho=({\sf sgn}\ q_3)\sqrt{q_2^2+q_3^2}$, \ ${\cos\xi}=\frac{q_2}{\sqrt{q_2^2+q_3^2}}$, \ ${\sin\xi}=\frac{|q_3|}{\sqrt{q_2^2+q_3^2}}$.
		
		Denote, $Z=\sigma_2{\cos\xi}+\sigma_3\sin\xi$. It is easy to see that $Z^2=\sigma_0$. Then
		$$
		e^Q=e^{\rho{Z}}=\sigma_0\cosh\rho+Z \sinh\rho,
		$$
		$$J_\mathcal{M}=\sigma_1e^Q=\sigma_{1}e^{\rho{Z}}=\sigma_{1}\cosh\rho-i\sigma_{1}\sinh\rho\sin\xi +i\sigma_{3}\sinh\rho\cos\xi.$$
		
		The operator $J_{\mathcal{M}}$ determines the
		subspaces $\mathcal{M}_\pm$ in the fundamental decomposition
		(\ref{uf2b}), specifically,
		\begin{equation}\label{uf101}
			\mathcal{M_+}=(I+J_\mathcal{M})\mathbb{C}^2, \quad
			\mathcal{M_-}=(I-J_\mathcal{M})\mathbb{C}^2.
		\end{equation}
		It anticommutes with $Q$ since $
		J_{\mathcal{M}}Q=\sigma_1e^QQ=\sigma_1Qe^Q=-Q\sigma_1e^Q=-QJ_{\mathcal{M}}$. Furthermore, $J_{\mathcal{M}}J_{\mathcal{L}}=\sigma_1e^Q\sigma_1=e^{-Q}$ and $J_{\mathcal{L}}J_{\mathcal{M}}=e^Q$.
	\end{example}
	
	At the end of the section, it is worth noting that
	the concept of non-decomposable spaces is less common in mathematical physics. Their distinguishing feature is the presence of an `enormous amount' of neutral vectors.
	To our current knowledge, the most straightforward example of a non-decomposable space was introduced by Savage \cite{Savage} (see also \cite[Example 11.3]{AK_Bognar}): the space $\mathcal{H}$ coincides with space of complex sequences $f=\{x_i\}_{-\infty}^\infty$ that are finite from the left: $x_i=0$ for $i\leq{i}(f)$ endowed with the indefinite inner product
	$$
	[f,g]=\sum\overline{x_i}{y}_{-i-1}, \qquad
	g=\{y_i\}_{-\infty}^\infty.
	$$
	Analyzing Savage's example carried out in \cite{AK_Azizov} yields the following assertion:
	\begin{thm}\label{uf4}
		If $\mathcal{H}$ admits the decomposition as a direct sum $\mathcal{H} = \mathcal{L}_1 \dot{+} \mathcal{L}_2$, where $\mathcal{L}_1$ and $\mathcal{L}_2$ are two non-isomorphic neutral linear subspaces of $\mathcal{H}$, then $\mathcal{H}$ is non-decomposable.
	\end{thm}
	Theorem \ref{uf4} facilitates the construction of various non-decomposable spaces. For instance, consider the space $\mathcal{H}$ , which consists of  functions $f$ defined on $[-a, a]$ such that  $f(x)$ is a polynomial for $-a\leq{x}<0$ and $f(x)\in{L_2(0,a)}$ for $x\geq{0}$.
	The space $\mathcal{H}$ endowed with  $\mathcal{PT}$ inner product
	$$
	[f, g]=\int_{-a}^a\overline{f(-x)}{g(x)}dx
	$$
	is non-decomposable because one can choose $\mathcal{L}_1$
	as the set of polynomials with supports in $[-a, 0]$ while $\mathcal{L}_2=L_2(0, a)$. This  example, along with many similar ones, suggests that a non-decomposable space can be extended to a decomposable one. In the remainder of this paper, as already stressed, we will focus exclusively on decomposable spaces.

	\subsection{The group of unitary operators in a Krein space $\mathcal{H}_{ext}$.}\label{aa1} Consider a one-parameter group of linear operators $U=\{U(t) : t\in\mathbb{R}\}$ determined by the formula (\ref{e4}).
	By construction, the operators $U(t)$ are unitary on a linear space $\mathcal{H}$ with respect to the indefinite inner product $[\cdot, \cdot]$, see (\ref{uf12}).  However, we cannot be certain that these operators are bounded in the pre-Hilbert space $(\mathcal{H}, \langle\cdot, \cdot\rangle_{\mathcal{L}})$.
	To analyze this problem we consider, for every $t\in\mathbb{R}$,  a fundamental decomposition
	\begin{equation}\label{uf15}
		\mathcal{H}=\mathcal{L}_+^t[\dot{+}]\mathcal{L}_-^t,
	\end{equation}
	where the subspace $\mathcal{L}_+^t=U(t)\mathcal{L}_+$ is positive,  while $\mathcal{L}_-^t=U(t)\mathcal{L}_-$ is negative and $\mathcal{L}_\pm$ represent the positive/negative subspaces from the decomposition (\ref{uf2}). In particular, since $U(0)=I$, we have $\mathcal{L}_\pm^0=\mathcal{L}_\pm$. This shows that (\ref{uf15}) can be seen as a {\em time evoluted version} of (\ref{uf2}).

	For each $t\not=0$, the decomposition (\ref{uf15}) determines, by analogy with (\ref{uf3}), an inner product $\langle\cdot,\cdot\rangle_t$ and a norm  $\|\cdot\|_t$ on $\mathcal{H}$. %In particular, the norm  $\|\cdot\|_{\mathcal{L}}$ defined by (\ref{uf3}) coincides with $\|\cdot\|_0$ and $(\mathcal{H}, \langle\cdot, \cdot\rangle_{\mathcal{L}})=(\mathcal{H}, \langle\cdot, \cdot\rangle_{0})$.
	The operator of fundamental symmetry $J_t$ associated with (\ref{uf15}) satisfies the relation
	\begin{equation}\label{uf20}
		J_{t}U(t)=U(t)J_{\mathcal{L}}, \end{equation}
	{where $J_{\mathcal{L}}$ is the fundamental symmetry for $t=0$.}
	\begin{lemma}\label{new1}
		The following are equivalent for every $t\not=0$:
		\begin{itemize}
			\item[(i)] operators $U(t)$ and $U^{-1}(t)$ are bounded in the pre-Hilbert space $(\mathcal{H}, \langle\cdot, \cdot\rangle_{\mathcal{L}})$;
			\item[(ii)] the norms $\|\cdot\|_{\mathcal{L}}$ and $\|\cdot\|_{t}$ are equivalent on $\mathcal{H}$.
		\end{itemize}
	\end{lemma}
	\begin{proof} $(i)\to(ii)$. Using (\ref{uf20}), we get
		$$
		\|J_t\|_{\mathcal{L}}=\|U(t)J_{\mathcal{L}}U^{-1}(t)\|_{\mathcal{L}}\leq\|U(t)\|_{\mathcal{L}}\|J_{\mathcal{L}}\|_{\mathcal{L}}\|U^{-1}(t)\|_{\mathcal{L}}=\|U(t)\|_{\mathcal{L}}\|U^{-1}(t)\|_{\mathcal{L}}<\infty
		$$
		i.e., the operator $J_t$ is bounded in $(\mathcal{H}, \langle\cdot, \cdot\rangle_{{\mathcal{L}}})$.
		By virtue of Theorem \ref{ups131b}, the norms $\|\cdot\|_{\mathcal{L}}$ and $\|\cdot\|_t$ are equivalent on $\mathcal{H}$.
		
		$(ii)\to(i)$. Applying (\ref{bu2}) with $J_{\mathcal{M}}=J_t$, and taking (\ref{uf12}) and (\ref{uf20}) into account, we obtain for every  $f\in\mathcal{H}$
		\begin{equation}\label{ups14}
			\|U(t)f\|_t^2=[J_tU(t)f, U(t)f]=[U(t)J_{\mathcal{L}}f, U(t)f]=[J_{\mathcal{L}}f, f]=\|f\|^2_{\mathcal{L}}.
		\end{equation}
		
		Since the norms $\|\cdot\|_{\mathcal{L}}$ and $\|\cdot\|_t$ are equivalent,  there exist constants
		$\alpha_i>0$ such that $\|\cdot\|_{\mathcal{L}}\leq\alpha_1\|\cdot\|_{t}$
		and $\|\cdot\|_{t}\leq\alpha_2\|\cdot\|_{\mathcal{L}}$. In this case, by virtue of (\ref{ups14}),
		$$
		\|U(t)f\|_{\mathcal{L}}\leq\alpha_1\|U(t)f\|_t=\alpha_1\|f\|_{\mathcal{L}}, \qquad  f\in\mathcal{H}.
		$$
		The obtained inequality yields the boundedness of $U(t)$ in  $(\mathcal{H}, \langle\cdot, \cdot\rangle_{{\mathcal{L}}})$.
		
		Setting $g=U(t)f$ in  (\ref{ups14}) and taking into account that  $U(t)$ is fully invertible on $\mathcal{H}$, we obtain
		$$
		\|U^{-1}(t)g\|_{\mathcal{L}}=\|g\|_t\leq\alpha_2\|g\|_{\mathcal{L}}, \qquad
		g\in\mathcal{H}.
		$$
		This means that $U^{-1}(t)$ is  bounded in $(\mathcal{H}, \langle\cdot, \cdot\rangle_{{\mathcal{L}}})$.
	\end{proof}
	
	By Lemma \ref{new1}, the boundedness of the operators in the group $U$ can be established by verifying the equivalence of the norms $\|\cdot\|_{\mathcal{L}}$ and $\|\cdot\|_{t}$.
	However, this condition is difficult to verify directly and can be reformulated using operators of fundamental symmetry.
	
	\begin{cor}\label{uf64b}
		If operators of fundamental symmetry $J_t$ associated with decomposition (\ref{uf15}) are bounded in the pre-Hilbert space $(\mathcal{H}, \langle\cdot, \cdot\rangle_{\mathcal{L}})$ for $t>0$, then all operators of the  group $U$ are bounded in  $(\mathcal{H}, \langle\cdot, \cdot\rangle_{{\mathcal{L}}})$.
	\end{cor}
	\begin{proof}
		Assuming the conditions of Corollary \ref{uf64b} are satisfied and using Theorem \ref{ups131b} with  $J_\mathcal{M}=J_t$ we derive the equivalence of the norms $\|\cdot\|_{\mathcal{L}}$ and $\|\cdot\|_t$ for $t>0$. By Lemma \ref{new1}, operators $U(t)$ and $U^{-1}(t)$ are bounded in $(\mathcal{H}, \langle\cdot, \cdot\rangle_{\mathcal{L}})$ for  $t>0$. Taking the relation $U^{-1}(t)=U(-t)$
		into account, we conclude the proof. 
	\end{proof}
	
	Consider the Hilbert space $(\mathcal{H}_{ext}, \langle\cdot, \cdot\rangle_{{\mathcal{L}}})$ obtained by completing $(\mathcal{H}, \langle\cdot, \cdot\rangle_{\mathcal{L}})$  with respect to the norm $\|\cdot\|_{\mathcal{L}}$. It should be noted that $(\mathcal{H}_{ext}, \langle\cdot, \cdot\rangle_{{\mathcal{L}}})$ can be also considered as a Krein space $(\mathcal{H}_{ext}, [\cdot,\cdot])$ with indefinite inner product $[\cdot, \cdot]=\langle{J_{\mathcal{L}}\cdot}, \cdot\rangle_{\mathcal{L}}$, where $J_{\mathcal{L}}$
	is the extension by continuity in
	$\mathcal{H}_{ext}$
	of the operator of fundamental symmetry $J_\mathcal{L}$ associated with (\ref{uf2}).

	Assume that the conditions of Corollary \ref{uf64b} are satisfied, i.e., all operators of fundamental symmetry $J_t$ ($t>0$) associated with (\ref{uf15}) are bounded in  $(\mathcal{H}, \langle\cdot, \cdot\rangle_{\mathcal{L}})$. Then operators $U(t)\in{U}$ can be extended by continuity to bounded operators acting in the Hilbert space $(\mathcal{H}_{ext}, \langle\cdot, \cdot\rangle_{{\mathcal{L}}})$. Denote by $U_{ext}=\{U(t) : t\in\mathbb{R}\}$ the corresponding group of operators. It follows from (\ref{uf12}) that each $U(t)\in{U}_{ext}$
	is a unitary operator in the Krein space $(\mathcal{H}_{ext}, [\cdot,\cdot])$.

	The following result clarifies the general form of operators that are unitary in the Krein space $(\mathcal{H}_{ext}, [\cdot,\cdot])$. An alternative description can be found in \cite[Theorem 5.10, Chapter II]{AK_Azizov}.

	\begin{prop}\label{uf50}
		
		Each operator $U(t)$ of the group $U_{ext}$ has the form
		\begin{equation}\label{ups47}
			U(t)=e^{-Q_t/2}Y_t, \qquad t\in\mathbb{R},
		\end{equation}
		where $Q_t$  is a bounded self-adjoint operator, and $Y_t$
		is a unitary operator, both defined in the Hilbert space
		$(\mathcal{H}_{ext}, \langle\cdot, \cdot\rangle_{\mathcal{L}})$ and such that
		$$
			J_{\mathcal{L}}Q_t=-Q_tJ_{\mathcal{L}}, \qquad  J_{\mathcal{L}}Y_t=Y_tJ_{\mathcal{L}}.
		$$
	\end{prop}
	\begin{proof} By setting $J_t = J_\mathcal{M}$ in Proposition \ref{ups1} and using relation (\ref{bu2}), we establish the existence of a bounded self-adjoint operator  $Q_t$ in $(\mathcal{H}_{ext}, \langle\cdot, \cdot\rangle_{{\mathcal{L}}})$ such that $J_{\mathcal{L}}Q_t=-Q_tJ_{\mathcal{L}}$ and $J_{\mathcal{L}}J_t=e^{Q_t}$. Furthermore,
		\begin{equation}\label{ups53}
			\|{f}\|_{t}^2=[J_tf,f]=\langle{J_{\mathcal{L}}J_tf, f\rangle}_{\mathcal{L}}=\langle{e^{Q_t}}f, f\rangle_{{\mathcal{L}}}=\|{e^{Q_t/2}f}\|_{\mathcal{L}}^2, \quad f\in\mathcal{H}_{ext}.
		\end{equation}
		It follows from (\ref{ups14}) and (\ref{ups53}),
		$$
		\|U(t)f\|_t=\|e^{Q_t/2}U(t)f\|_{\mathcal{L}}=\|f\|_{\mathcal{L}}, \qquad f\in\mathcal{H}_{ext}.
		$$
		This means that the operator
		$Y_t=e^{Q_t/2}U(t)$ is {an isometry} operator  in the Hilbert space  $(\mathcal{H}_{ext}, \langle\cdot,\cdot\rangle_{\mathcal{L}})$.
		{Moreover, given that
			$U(t)$ is unitary in the Krein space
			$(\mathcal{H}_{ext}, [\cdot,\cdot])$ and that
			$e^{Q_t/2}$  has a bounded inverse, we conclude that
			$Y$ maps $\mathcal{H}_{ext}$
			onto itself. Therefore, $Y_t$ is a unitary operator  in the Hilbert space  $(\mathcal{H}_{ext}, \langle\cdot,\cdot\rangle_{\mathcal{L}})$}.

		Additionally, referring back to
		(\ref{uf20}),
		$$
		J_{\mathcal{L}}Y_t=e^{-Q_t/2}J_{\mathcal{L}}U(t)=e^{-Q_t/2}J_{\mathcal{L}}J_tU(t)J_{\mathcal{L}}=e^{-Q_t/2}e^{Q_t}U(t)J_{\mathcal{L}}=e^{Q_t/2}U(t)J_{\mathcal{L}}=Y_tJ_{\mathcal{L}}.
		$$
		that completes the proof. 
	\end{proof}
	
	\subsection{The choice of a proper definite inner product.}
	
	It is crucial to emphasize that selecting the appropriate decomposition (\ref{uf2}) of $\mathcal{H}$ is principal to our research. This choice, in conjunction with the indefinite inner product $[\cdot, \cdot]$, determines the definite inner product $\langle\cdot, \cdot\rangle_{\mathcal{L}}$ and, consequently, the Hilbert space $(\mathcal{H}_{ext}, \langle\cdot, \cdot\rangle_{\mathcal{L}})$.
	What would happen if one chose a different decomposition: $\mathcal{H}={\mathcal{L}}_+[\dot{+}]{\mathcal{L}}_-$ onto positive ${\mathcal{L}}_+$ and negative
	${\mathcal{L}}_-$ subspaces
	than the one in (\ref{uf2})? In this case, a new family of fundamental symmetry operators
	${J}_t$ associated with (\ref{uf15}) would be obtained and
	several possibilities arise.
	
	We start by considering a worst-case scenario, assuming that at least one of the fundamental symmetries $J_t$ $(t>0)$ is an \emph{unbounded} operator in the pre-Hilbert space $(\mathcal{H}, \langle\cdot, \cdot\rangle_{{\mathcal{L}}})$.
	Then, in view of Theorem \ref{ups131b} and Lemma \ref{new1}, the one-parameter group $U$ will include at least one unbounded operator in the pre-Hilbert space
	$(\mathcal{H}, \langle\cdot, \cdot\rangle_{{\mathcal{L}}})$.
	This outcome suggests that the chosen decomposition
	$\mathcal{H}={\mathcal{L}}_+[\dot{+}]{\mathcal{L}}_-$
	may not have been appropriate for the construction of one-parametric group of bounded operators in $(\mathcal{H}_{ext}, \langle\cdot, \cdot\rangle_{\mathcal{L}})$.

	If operators ${J}_t$ $(t>0)$ are \emph{bounded} in  $(\mathcal{H}, \langle\cdot, \cdot\rangle_{{\mathcal{L}}})$,
	then the operator group $U$ can be realized as a group of bounded operators in the Hilbert space
	$(\mathcal{H}_{ext}, \langle\cdot, \cdot\rangle_{{\mathcal{L}}})$. This specific case was discussed in Subsection \ref{aa1}.
	
	The following result {deals with a significant subclass of uniformly bounded operators}  $J_t$, which allows the group $U_{ext}$ to be realized as a group of unitary operators in a Hilbert space $\mathcal{H}_{ext}$.
	
	\begin{thm}\label{uf53}
		Let the decomposition (\ref{uf2}) of $\mathcal{H}$ be chosen in such a way that the operators of fundamental symmetry $J_t$  are \emph{uniformly bounded} in the pre-Hilbert space $(\mathcal{H}, \langle\cdot, \cdot\rangle_{\mathcal{L}})$, i.e., there exists $c>0$ such that
		\begin{equation}\label{uman8}
			\|J_tf\|_{\mathcal{L}}<c\|f\|_{\mathcal{L}}, \qquad t>0, \quad  f\in\mathcal{H}.
		\end{equation}
		Then there exists a decomposition
		$\mathcal{H}=\mathcal{M}_+[\dot{+}]\mathcal{M}_-$
		as in (\ref{uf2b}), where the subspaces
		$\mathcal{M}_\pm$
		are invariant under the operators of the group $U$. The norm $\|\cdot\|_{\mathcal{M}}$
		defined by (\ref{uf2b}) is equivalent to the norm $\|\cdot\|_{\mathcal{L}}$
		on $\mathcal{H}$, and the completion of the operators from
		$U$ in the Hilbert space $(\mathcal{H}_{ext}, \langle\cdot, \cdot\rangle_{\mathcal{M}})$
		results in a group of unitary operators $U_{ext}$.
	\end{thm}
	\begin{proof}
		Corollary \ref{uf64b} and (\ref{uman8}) imply that the group $U$ can be extended by continuity to
		the  group $U_{ext}=\{U(t) : t\in\mathbb{R}\}$
		of bounded operators  in the Hilbert space $(\mathcal{H}_{ext}, \langle\cdot, \cdot\rangle_{\mathcal{L}})$. Each of these operators $U(t)$ has the form (\ref{ups47}), where
		the operator $Q_t$ is defined by the relation $J_tJ_{\mathcal{L}}=e^{-Q_t}$. In view of
		(\ref{uman8}),
		$$
		\|e^{-Q_t}f\|_{\mathcal{L}}=\|J_tJ_{\mathcal{L}}f\|_{\mathcal{L}}<c\|J_{\mathcal{L}}f\|_{\mathcal{L}}=c\|f\|_{\mathcal{L}}, \qquad f\in\mathcal{H}, \quad t>0.
		$$
		Therefore, $\|e^{-Q_t}\|_{\mathcal{L}}\leq{c}$ and
		$$
		\|U(t)f\|_{\mathcal{L}}=\|e^{-Q_t/2}Y_tf\|_{\mathcal{L}}\leq\|e^{-Q_t/2}\|_{\mathcal{L}}\|Y_tf\|_{\mathcal{L}}\leq\sqrt{c}\|Y_tf\|_{\mathcal{L}}=\sqrt{c}\|f\|_{\mathcal{L}}, \qquad f\in\mathcal{H}, \quad t>0.
		$$
		This means that the norms of $U(t)$ in $(\mathcal{H}_{ext}, \langle\cdot, \cdot\rangle_{\mathcal{L}})$ are uniformly bounded for $t>0$, i.e., $\|U(t)\|_{\mathcal{L}}\leq\sqrt{c}$. Since, for all $t\in\mathbb{R}$, the operator  $U(t)$ is  unitary in the Krein space $(\mathcal{H}_{ext}, [\cdot,\cdot])$, the following relation holds
		$$
		U(-t)=U^{-1}(t)=J_{\mathcal{L}}U^\dag(t)J_{\mathcal{L}}, \qquad t>0,
		$$
		where $U^\dag(t)$ is the adjoint operator of $U(t)$ in the Hilbert space $(\mathcal{H}_{ext}, \langle\cdot, \cdot\rangle_{\mathcal{L}})$.
		Hence, $\|U(-t)\|_{\mathcal{L}}=\|J_{\mathcal{L}}U^\dag(t)J_{\mathcal{L}}\|_{\mathcal{L}}\leq\|U^\dag(t)\|_{\mathcal{L}}=\|U(t)\|_{\mathcal{L}}$ for $t>0$. The obtained relation means that the operators $U(t)$, $t\in\mathbb{R}$ are uniformly bounded in $(\mathcal{H}_{ext}, \langle\cdot, \cdot\rangle_{\mathcal{L}})$.
		In this case (see \cite[Theorem 5.18, Chapter II]{AK_Azizov}), there exists a decomposition
		\begin{equation}\label{uman9}
			\mathcal{H}_{ext}={\mathcal{M}}_+^{ext}[\dot{+}]{\mathcal{M}}_-^{ext}
		\end{equation}
		onto orthogonal (with respect to $[\cdot,\cdot]$) sum of positive/negative subspaces  ${\mathcal{M}}_\pm^{ext}$ that is invariant with respect to the group  $U_{ext}$, i.e., $U(t) : {\mathcal{M}}_+^{ext} \to {\mathcal{M}}_+^{ext}$ and $U(t) : {\mathcal{M}}_-^{ext} \to {\mathcal{M}}_-^{ext}$ for all $t\in\mathbb{R}$. Since the operators $U(t)$ are unitary in the Krein space $(\mathcal{H}_{ext}, [\cdot,\cdot])$, their restriction to ${\mathcal{M}}_+^{ext}$ and ${\mathcal{M}}_-^{ext}$ results in unitary operators acting within the Hilbert spaces
		$({\mathcal{M}}_+^{ext}, [\cdot, \cdot])$ and $({\mathcal{M}}_-^{ext}, -[\cdot, \cdot])$, respectively.
		
		Consider a new definite inner product on $\mathcal{H}_{ext}$, defined as $\langle\cdot,\cdot\rangle_{{\mathcal{M}}}=[J_{\mathcal{M}}\cdot,\cdot]$, where $J_{\mathcal{M}}$ is an operator of fundamental symmetry associated with the decomposition (\ref{uman9}).
		The norms $\|\cdot\|_{\mathcal{L}}$ and $\|\cdot\|_{\mathcal{M}}$ are equivalent on $\mathcal{H}_{ext}. $  Furthermore, the Hilbert spaces
		$({\mathcal{M}}_+^{ext}, [\cdot, \cdot])$ and $({\mathcal{M}}_-^{ext}, -[\cdot, \cdot])$ considered above coincide with the orthogonal subspaces $({\mathcal{M}}_+^{ext}, \|\cdot\|_{\mathcal{M}})$ and $({\mathcal{M}}_-^{ext}, \|\cdot\|_{\mathcal{M}})$ of the Hilbert space $(\mathcal{H}_{ext}, \langle\cdot, \cdot\rangle_{\mathcal{M}})$. This means that each operator $U(t)\in{U}_{ext}$ is the orthogonal sum of two unitary operators acting in $({\mathcal{M}}_+^{ext}, \|\cdot\|_{\mathcal{M}})$ and $({\mathcal{M}}_-^{ext}, \|\cdot\|_{\mathcal{M}})$, respectively. Hence, $U_{ext}$ is a group of unitary operators in the Hilbert space $(\mathcal{H}_{ext}, \langle\cdot, \cdot\rangle_{\mathcal{M}})$.
		
		To complete the proof, we note that the subspaces ${\mathcal{M}}_\pm={\mathcal{M}}_\pm^{ext}\cap\mathcal{H}$  are positive and negative, respectively, and that the decomposition
		(\ref{uf2b}) holds.
		Furthermore, these subspaces
		are  invariant under the action of $U(t)$ since $U(t)$ are fully  invertible operators on  $\mathcal{H}$.
	\end{proof}
	
	The following result is a direct consequence of Theorem \ref{uf53}
	\begin{cor}
		Let the conditions of Theorem \ref{uf53} be satisfied, and let
		$H$ be the infinitesimal generator of the group
		$U_{ext}$. Then
		$iH$ is a self-adjoint operator in the Hilbert space
		$(\mathcal{H}_{ext},
		\langle\cdot, \cdot \rangle_{\mathcal{M}})$.
	\end{cor}

	\section{Examples}
	
	This section is entirely devoted to some concrete realizations of the abstract framework described in the previous sections.
	
	\subsection{A dilation group.}\label{s4.1}
	Consider a one-parameter group of operators
	$$
	(U(t)f)(x)=e^{t/2}f(e^tx), \quad  t\in\mathbb{R}
	$$
	acting in  an indefinite inner product space $\mathcal{H}=\mathcal{S}(\mathbb{R})$ with $\mathcal{PT}$ inner product
	(\ref{uman4}) (see Example \ref{uman10} for details).
	
	The space
	${\mathcal{H}}$ is  decomposable and the
	positive $\mathcal{L}_+$ and the negative $\mathcal{L}_-$ subspaces in the decomposition (\ref{uf2}) of $\mathcal{H}$ can be chosen as   the subspaces of even functions (for $\mathcal{L}_+$) and odd functions (for $\mathcal{L}_-$).
	The operator of fundamental symmetry $J_{\mathcal{L}}$ associated with  (\ref{uf2}) coincides with
	the space parity operator $\mathcal{P}f(x)=f(-x)$ and the definite inner product $\langle\cdot, \cdot\rangle_{\mathcal{L}}=[{\mathcal{P}}\cdot, \cdot]$ is the ordinary $L_2$-inner product
	$\langle{f}, g\rangle_{\mathcal{L}}=\int_{\mathbb{R}}\overline{f(x)}{g(x)}dx$.

	The subspaces $\mathcal{L}_{\pm}$ are invariant with respect to $U(t)$.  Therefore the subspaces $\mathcal{L}_\pm^t$ in the decomposition (\ref{uf15}) coincide with $\mathcal{L}_\pm$ and $J_{t}=\mathcal{P}$ for all $t>0$. Hence the operators $J_t$ are uniformly bounded and the subspaces $\mathcal{M}_{\pm}$ indicated in Theorem \ref{uf53} coincide with $\mathcal{L}_{\pm}$. The latter means that
	the definite inner product
	$\langle\cdot,\cdot\rangle_{{\mathcal{M}}}$
	coincides with the ordinary $L_2$-inner product and $(\mathcal{H}_{ext}, \langle\cdot, \cdot\rangle_{\mathcal{M}})=L_2(\mathbb{R})$.
	
	To summarize, extending the operators
	$U(t)$ initially defined on $\mathcal{S}(\mathbb{R})$
	by continuity onto
	$L_2(\mathbb{R})$  results in the formation of a group
	$U_{ext}$
	of unitary operators acting on $L_2(\mathbb{R})$. The infinitesimal generator of $U_{ext}$
	is a self-adjoint operator in $L_2(\mathbb{R})$ and it coincides with  the Berry-Keating Hamiltonian $H_{BK}=\frac{-1}{2}(XP+PX)$, where $X$ and $P$ are usual self-adjoint position and momentum operators
	\cite{BK}, \cite[p.223]{T}.

	\subsection{Generalized quantum harmonic oscillator.}\label{4.2}
	Consider a  linear space
	$
	\mathcal{H}=\{f(x)=P(x)e^{-x^2/2} , \ x\in\mathbb{R}\},
	$
	where $P(x)$ is an arbitrary polynomial.
	The space ${\mathcal{H}}$ endowed with $\mathcal{PT}$ inner product (\ref{uman4}) is decomposable. Similar to the previous example,  the
	subspaces $\mathcal{L}_\pm$ in  (\ref{uf2})  can be chosen as  the subspaces of even and odd functions.
	The operator of fundamental symmetry $J_{\mathcal{L}}$  coincides with
	the space parity operator $\mathcal{P}$ and  $\langle{f}, g\rangle_{\mathcal{L}}=\int_{\mathbb{R}}\overline{f(x)}{g(x)}dx$. Notice that the linear space $\mathcal{H}$ we are considering here is just a subspace of the one used in the previous example, $\mathcal{S}(\mathbb{R})$, and therefore it is not a surprise that some of the preliminaries above are repeated here.
	
	The  Hermite functions \begin{equation}\label{aaa1}
		f_n(x)=\frac{1}{\sqrt{2^nn!\sqrt{\pi}}}H_n(x)e^{-x^2/2},  \quad H_n(x)=e^{x^2/2}(x-\frac{d}{dx})^ne^{-x^2/2}
	\end{equation}
 {are elements of the space $\mathcal{H}$, satisfy the relation $\mathcal{P}f_n=(-1)^nf_n$, and form an orthogonal sequence with respect to the inner product $[\cdot, \cdot]$, such that}
	$[f_n, f_n]=\langle\mathcal{P}f_n, f_n\rangle_{\mathcal{L}}={(-1)^n\langle{f}_n, f_n\rangle_{\mathcal{L}}}=(-1)^n.$
	Furthermore,
	$\mathcal{H}={\sf span}\ \{f_{n}\}_{n=0}^\infty.$ This means that the formula
	$$
	W(t)f_n=e^{\lambda_nt}f_n, \qquad \lambda_n\in\mathbb{C}, \quad t\geq{0},
	\quad n\in\mathbb{N}\cup\{0\}.
	$$
	determines a one-parameter semi-group $W_+$ of linear operators on $\mathcal{H}$.
	
	{According to Example \ref{uman63}, the operators $W(t)$ map $\mathcal{F}_{++}$ onto itself in a one-to-one manner if and only if the condition (\ref{uman1963}) holds. In this case (see again Example \ref{uman63}) }  
	the formula
	$$
	U(t)f_n=e^{i{\sf Im} \lambda_nt}f_n, \quad t\in\mathbb{R}.
	$$
	defines a group of unitary operators on $\mathcal{H}$ with respect to $[\cdot, \cdot]$.

	The subspaces $\mathcal{L}_\pm$ introduced above can be characterized as follows:
	$\mathcal{L}_+={\sf span}\ \{f_{2k}\}_{k=0}^\infty$ and $\mathcal{L}_-={\sf span}\ \{f_{2k+1}\}_{k=0}^\infty$. This means that $\mathcal{L}_{\pm}$ are invariant with respect to $U(t)$.
	Similar to the previous example, it can be stated that
	$\mathcal{L}_{\pm}=\mathcal{L}_{\pm}^t$ and $J_{t}=\mathcal{P}$ for all $t>0$. Hence the operators $J_t$ are uniformly bounded and the subspaces $\mathcal{M}_{\pm}$ indicated in Theorem \ref{uf53} coincide with $\mathcal{L}_{\pm}$. This means that
	the Hilbert space $(\mathcal{H}_{ext}, \langle\cdot, \cdot\rangle_{\mathcal{M}})$ coincides with
	$L_2(\mathbb{R})$ and  the completion of the operators $U(t)$  in $L_2(\mathbb{R})$
	results in the formation of a group $U_{ext}$ of unitary operators acting on $L_2(\mathbb{R})$. This group coincides with the group generated by  the  harmonic oscillator
	$H=-\frac{d^2}{dx^2}+x^2$ when
	$\lambda_n=i(2n+1)$, since $e^{iHt}f_n=e^{i(2n+1)t}f_n$, $\forall n\geq0$.
	
	\vspace{2mm}
	
	{\bf Remark:--} It is not hard to generalize this example to other cases: if we consider a self-adjoint Hamiltonian with even potential we can look for odd or even eigenstates. If the set of these eigenstates is an orthonormal basis for a certain Hilbert space $\mathcal{H}$, we can repeat a similar construction as above, using the same $\mathcal{PT}$ inner product (\ref{uman4}). This is what happens, for instance, for a particle in a finite potential well, defined on ${{L}}_2(-a,a)$, $a>0$, \cite{merz}.

\subsection{Quantum harmonic oscillator shifted in the complex plain.}\label{4.3a}
Since the Hermite functions (\ref{aaa1}) are entire functions on $\mathbb{C}$, their complex shift can be defined:
$$
 \phi_n(x)=f_n(x+ia),  \qquad  a\in\mathbb{R}\setminus\{0\}, \quad n=0,1,2,\ldots
$$

Consider  a  linear space  
$
\mathcal{H}={\sf span} \{\phi_n\}_{n=0}^\infty
$ 
endowed with $\mathcal{PT}$ inner product $[\cdot, \cdot]$, see (\ref{uman4}). The sequence $\{\phi_n\}_{n=0}^\infty$ is orthogonal with respect to $[\cdot, \cdot]$ and $[\phi_n, \phi_n]=(-1)^n$ \cite{Mit}.

Similarly to subsection \ref{4.2} it can be established that the formula
$$
U(t)\phi_n=e^{i{\sf Im} \lambda_nt}\phi_n, \qquad \lambda_n\in\mathbb{C}, \quad t\in\mathbb{R}.
	$$
	defines a group $U$ of unitary operators on $\mathcal{H}$ with respect to the $\mathcal{PT}$ inner product (\ref{uman4}). 

Next, we should aim to identify a suitable inner product that allows us to interpret 
$U$
 as a set of unitary operators in a Hilbert space.
 To achieve this, we need to find an appropriate decomposition (\ref{uf2}) of 
$\mathcal{H}$
 into the positive and negative subspaces. In contrast to subsection \ref{4.2}, these subspaces 
$\mathcal{L}_\pm$
  cannot be chosen as the subspaces of even and odd functions, since the even and odd components of 
$\phi_n$  do not belong to 
$\mathcal{H}$. This implies that the standard 
$L_2$-inner product
$\langle{f}, g\rangle_{L_2}=\int_{\mathbb{R}}\overline{f(x)}{g(x)}dx$
on $\mathcal{H}$ \emph{cannot be determined} by the $\mathcal{PT}$-inner product (\ref{uman4}) using the methods developed in subsection \ref{s3.1}.

A natural choice for $\mathcal{L}_\pm$
arises from the observation that $[\phi_n, \phi_m]=(-1)^n\delta_{nm}$.
 Following this approach, we select the decomposition (\ref{uf2}) of 
$\mathcal{H}$ into positive and negative subspaces: $\mathcal{L}_+={\sf span}\ \{\phi_{2k}\}_{k=0}^\infty$ and $\mathcal{L}_-={\sf span}\ \{\phi_{2k+1}\}_{k=0}^\infty$. 

Consider a Hilbert space $(\mathcal{H}_{ext}, \langle\cdot, \cdot\rangle_{\mathcal{L}})$,  where the inner product $\langle\cdot, \cdot\rangle_{\mathcal{L}}$
is defined by (\ref{uf3}) and 
$\mathcal{H}_{ext}$
 denotes the completion of $\mathcal{H}$ 
 with respect to the norm $\|\cdot\|_{\mathcal{L}}=\sqrt{\langle\cdot, \cdot\rangle_{\mathcal{L}}}$.
Similarly to the previous example, it can be shown that completing 
$U(t)$ in $(\mathcal{H}_{ext}, \langle\cdot, \cdot\rangle_{\mathcal{L}})$, results in a one-parameter group $U_{ext}$ of unitary operators.  
This group coincides with the group generated by the shifted harmonic oscillator
	$H=-\frac{d^2}{dx^2}+x^2+2iax$ when
	$\lambda_n=i(2n+1+a^2)$, since $e^{iHt}\phi_n=e^{i(2n+1+a^2)t}\phi_n$, $\forall n\geq0$.

Let us specify the inner product  $\langle\cdot, \cdot\rangle_{\mathcal{L}}$  by applying the Fourier transform $$(F\phi)(\delta)=\frac{1}{\sqrt{2\pi}}\int_{-\infty}^\infty{e^{-i\delta{x}}\phi(x)}dx$$
to $\phi_n$. Consequently, $(F\phi_n)(\delta)=e^{-a\delta}(Ff_n)(\delta)$ and $\phi_n=F^{-1}e^{-a\delta}Ff_n$. Denote 
$$
A=F^{-1}e^{-a\delta}F\mathcal{P}F^{-1}e^{a\delta}F.
$$
 Taking into account that $\mathcal{P}f_n=(-1)^nf_n$, where $\mathcal{P}$ is the space parity operator, we conclude that 
$A\phi_n=F^{-1}e^{-a\delta}F\mathcal{P}f_n=(-1)^n\phi_n$.
 Therefore, $A$ is the operator of fundamental symmetry associated with (\ref{uf2}), i.e., $J_{\mathcal{L}}=A$.
{Since 
$\mathcal{P}$ commutes with $F$, the expression for $A$ can be rewritten as $$
J_{\mathcal{L}}=A=F^{-1}e^{-a\delta}\mathcal{P}e^{a\delta}F=F^{-1}\mathcal{P}e^{2a\delta}F=\mathcal{P}F^{-1}e^{2a\delta}F.
$$}
By virtue of (\ref{bu1}), for $\phi, \psi\in\mathcal{H}$,
$$
\langle{\phi}, \psi\rangle_{\mathcal{L}}=[J_{\mathcal{L}}\phi, \psi]=\langle\mathcal{P}J_{\mathcal{L}}\phi, \psi\rangle_{{L_2}}=\langle{F}^{-1}e^{2a\delta}F\phi, \psi\rangle_{{L_2}}=\langle{e}^{2a\delta}F\phi, F\psi\rangle_{{L_2}}.
$$
This means that the norms $\|\cdot\|_{\mathcal{L}}$ and $\|\cdot\|_{L_2}$ are not equivalent on $\mathcal{H}$ and the Hilbert space $(\mathcal{H}_{ext}, \langle\cdot, \cdot\rangle_{\mathcal{L}})$ does not coincide with $L_2(\mathbb{R})$.

{\bf Remark:--} Results similar to those in subsection \ref{4.3a} can be derived by examining a broad class of generalized Riesz systems that are orthogonal with respect to an indefinite inner product  \cite{FBSK}.

	\section{Conclusions}
	The initial problem of interest was the following: can the properties of a linear operator $W$ in an indefinite inner product space be uniquely determined by its restriction to the nonlinear set $\mathcal{F}_{++}$ of positive vectors? In subsections \ref{s2.1}, \ref{s2.2}, we provide an affirmative answer to this question. In particular, we show that the bijectivity of $W$ on $\mathcal{F}_{++}$ is equivalent to it being {\em close} to a unitary operator with respect to the indefinite inner product $[\cdot,\cdot]$ (Theorem \ref{ups30}).

	Starting from subsection \ref{2.3}, we focus on one-parameter semi-groups of operators $W_+ = \{W(t) : t \geq 0\}$, where each $W(t)$ maps $\mathcal{F}_{++}$ onto itself in a one-to-one manner. Under this natural restriction, we demonstrate that the semi-group $W_+$ can be transformed into a one-parameter group $U = \{U(t) : t\in\mathbb{R}\}$ of operators that are unitary with respect to the indefinite inner product on $\mathcal{H}$. Moreover, by imposing  additional conditions {in Section \ref{3}}, we show how to construct a suitable definite inner product $\langle\cdot, \cdot\rangle$ based on the indefinite inner product $[\cdot, \cdot]$, ensuring the unitarity of the operators $U(t)$ in the Hilbert space obtained by completing $\mathcal{H}$ with respect to $\langle\cdot, \cdot\rangle$ {(Theorem \ref{uf53})}.
	
	The last part of the paper is devoted to some physically relevant examples, meant to clarify our general settings in some concrete situation.
	
In our view, the results presented in {in Section \ref{s2}} 
 warrant special attention. {For example,} Lemma \ref{uf5b} specifically implies that the discrete spectrum of an operator, corresponding to { negative eigenvectors $f$, i.e. $[f,f]<0$, is uniquely determined by its behavior on positive vectors. 
 Another interesting result, we believe, is the semi-group $W_+$  discussed in Example \ref{uman63}. Its infinitesimal generator 
$$
Hf=\lim_{t\to+0}\frac{W(t)-I}{t}f
$$ 
can
be calculated for all\footnote{Because of the specific nature of the space 
$\mathcal{H}$ in Example \ref{uman63}, we do not need to know which norm is defined on $\mathcal{H}$.} $f\in\mathcal{H}$ 
and $H$ has eigenvectors $f_n$ 
  corresponding to eigenvalues 
\begin{equation}\label{uman9b}
\lambda_n=\frac{\alpha}{2}+i{\sf Im}\lambda_n, \qquad \alpha>0.
\end{equation}
This result can be easily extended to a more general case. Consider a one-parameter semi-group 
$W_+$ acting on an indefinite inner product space 
$\mathcal{H}$ and mapping 
$\mathcal{F}_{++}$ bijectively onto itself. 
Assume that the infinitesimal generator 
$H$
of $W_+$ exists. Then, each point $\lambda_n$
in the discrete spectrum of $H$ takes the form (\ref{uman9b}), where the constant $\alpha$ does not depend on the choice of $\lambda_n$.

In other words, the bijectivity of the operators in $W_+$  on the set of positive vectors implies that the discrete spectrum of the corresponding infinitesimal generator $H$
consists of points lying along a straight line parallel to the imaginary axis. This is of course not far from what is found and discussed in connection with the non trivial zeros of the Riemann zeta function $\xi(s)$, \cite{Edwards}. This is connected, we believe, to the recent interest of suitable quantum Hamiltonians in connection with these zeros, originally proposed in \cite{BK} (see also subsection \ref{s4.1}), and then explored by many authors using ideas often strongly connected to $\mathcal{PT}$-quantum mechanics. We only cite here some recent contributions, like \cite{bbm,kalauni}, together as some critical contributions, as the paper in \cite{bellisard}, where a series of mathematical aspects to be clarified in \cite{bbm} are pointed out. 

	\section*{Acknowledgments} {S. K. expresses his gratitude to the Dipartimento di Ingegneria of the University of Palermo for its financial support via its program of Networking.} S. K.  thanks Dr. Uwe Guenther for the valuable discussions and constructive criticism. F. B. acknowledges partial support from Palermo University and from G.N.F.M. of the INdAM. F, B. also  acknowledges partial support from PRIN Project {\em  Transport phonema in low dimensional structures: models, simulations and theoretical aspects}.

\end{document}